\documentclass[11pt,leqno]{amsart}
\makeatletter
\usepackage{amsfonts,delarray,amssymb,amsmath,amsthm,a4,a4wide}
\usepackage{color}

\title[Second Inner Variations and Local Minimizers]{On the second inner variations of Allen-Cahn type energies and applications to local minimizers}
\author{Nam Q. Le$^{*}$}
\thanks{$^{*}$ Department of Mathematics, Indiana University, Bloomington, 831 E 3rd St,
Bloomington, IN 47405, USA. Telephone: (1) 812-855-8538 \\ and Institute of Mathematics, Vietnam Academy of Science and Technology, 18 Hoang Quoc Viet, Hanoi, Vietnam
\newline {\it Email adress:} nqle@indiana.edu}
\newcommand{\review}[2][\right]{\relax
\ifx#1\right\relax \left.\fi#2#1\rvert}
\let\abs=\envert
 \newtheorem{definition}{Definition}[section]
\newtheorem{theorem}{Theorem}[section]
\newtheorem{claim}{Claim}[section]
\newtheorem{propo}{Proposition}[section]
\newtheorem{remark}{Remark}[section]
\newtheorem{cor}{Corollary}[section]
\newtheorem{lemma}{Lemma}[section]
\newcommand{\bef}{\begin{flushright}}
\newcommand{\eef}{\end{flushright}}

\newcommand{\eval}[2][\right]{\relax
\ifx#1\right\relax \left.\fi#2#1\rvert}

\let\abs=\envert

\numberwithin{equation}{section}

\newcommand\e{\varepsilon}  
\newcommand{\h}{\hspace*{.24in}}

\def\h{\hspace*{.24in}}

\def\beq{\begin{eqnarray*}}
\def\eeq{\end{eqnarray*}}

\def\RR{\mbox{$I\hspace{-.06in}R$}}

\newenvironment{myindentpar}[1]%
{\begin{list}{}%
         {\setlength{\leftmargin}{#1}}%
         \item[]%
}
{\end{list}}

\begin{document}

\maketitle
\begin{abstract}
In this paper, we obtain an explicit formula for the discrepancy between the limit of 
the second inner variations of $p$-Laplace Allen-Cahn energies and the second inner variation of their $\Gamma$-limit which is the area 
functional. Our analysis explains the mysterious discrepancy term found in our previous paper  \cite{Le} in the case $p=2$. The discrepancy 
term turns out to be related to the convergence 
of certain 4-tensors which are absent in the usual Allen-Cahn functional.  These (hidden) 4-tensors suggest that, in the complex-valued Ginzburg-Landau setting, we 
should expect a different discrepancy term which we are able to identify. Along the way, we partially answer a question of Kohn and Sternberg \cite{KS} by giving a relation between 
the limit of second variations of the Allen-Cahn functional and the second inner variation of the area functional at local minimizers. Moreover, our analysis reveals an interesting 
identity connecting second inner variation and Poincar\'e inequality for area-minimizing surfaces with volume constraint in the work of Sternberg and Zumbrun \cite{SZ2}.

\medskip

\begin{center}
{\bf R\'esum\'e}
\end{center}
Dans cet article, nous obtenons une formule explicite pour la diff\'erence entre la limite 
des deuxi\`emes variations internes des \'energies du p-Laplacien de Allen- Cahn et la seconde variation interne de leur $\Gamma$ -limite qui est la
fonctionnelle d'aire. Notre analyse explique la diff\'erence myst\'erieuse trouv\'ee dans notre article pr\'ec\'edent \cite{Le} dans les cas $ p = 2$. Cette diff\'erence se r\'ev\`ele \^etre en rapport avec la convergence
de certains 4-tenseurs qui sont absents dans la fonctionnelle Allen - Cahn habituelle. Ces 4 - tenseurs (cach\'es) sugg\`erent que, dans le cadre de Ginzburg- Landau \`a valeurs complexe, nous
nous devons attendre \`a un terme de divergence diff\'erent que nous sommes en mesure d'identifier. En particulier, nous r\'epondons en partie une question de Kohn et Sternberg \cite{KS} 
en donnant une relation entre
la limite des deuxi\`emes variations de la fonctionnelle Allen- Cahn et la deuxi\`eme variation interne de la fonctionnelle  d'aire aux points de minimum locaux. De plus, notre analyse 
r\'ev\`ele une identit\'e int\'eressante
qui relie les deuxi\`emes variations internes et l'in\'egalit\'e de Poincar\'e pour les surfaces d'aire minimisante avec contrainte de volume dans le travail de 
Sternberg et 
Zumbrun \cite{SZ2}.


\end{abstract}
\pagenumbering{arabic}

\section{Introduction and statement of the main results}   

This paper is concerned with the relationship between the second {\it variations, inner variations }of Allen-Cahn type energies and their Gamma-limits together with applications to local minimizers and Poincar\'e inequality. The main results of the paper are Theorems \ref{thm-ACp}, \ref{isolate2}, \ref{isolate3} and \ref{mainSV}.

The typical functionals we consider are of the form
$$A(u):= \int_{\Omega} F(u(x), \nabla u(x)) dx$$
where $\Omega$ is an open smooth bounded domain in $\RR^{N}$ ($N\geq 2$) and $F: \RR\times \RR^N\rightarrow \RR$ is a smooth function. We recall that the first and second (usual) variations of $A$ at $u\in C^{2}(\Omega)$ with respect to $\varphi \in C_{c}^{1}(\Omega)$, denoted by $dA(u,\varphi)$ and $d^2 A(u,\varphi)$ respectively, are defined by 
$$dA(u,\varphi)= \left.\frac{d}{dt}\right\rvert_{t=0} A(u + t\varphi),~d^{2} A(u,\varphi) = \left.\frac{d^2}{dt^2}\right\rvert_{t=0} A(u + t\varphi).$$
On the other hand, we can deform the domain $\Omega$ using velocity and acceleration vector fields $\eta, \zeta\in (C_{c}^{1}(\Omega))^{N}$. In fact, for $t$ sufficiently small, the map 
\begin{equation}\Phi_{t} (x) = x + t\eta(x) + \frac{t^2}{2}\zeta(x)
\label{map-deform}
\end{equation} is a diffeomorphism of $\Omega$ into itself. The first and second inner variations of $A$  at $u$ with respect to the velocity and acceleration vector fields $\eta$ and $\zeta$, denoted by $\delta A(u,\eta,\zeta)$ and $\delta^{2} A(u,\eta,\zeta)$ respectively, are defined by 

\begin{equation*}
\delta A(u,\eta,\zeta) = \left.\frac{d}{dt}\right\rvert_{t=0} A(u\circ\Phi_t^{-1}),~
 \delta^{2} A(u,\eta,\zeta) = \left.\frac{d^2}{dt^2}\right\rvert_{t=0} A(u\circ\Phi_t^{-1}).
\end{equation*}
The relationship between these two notions of variations will be clarified in Proposition \ref{all_var}.\\
{\bf Notation.} We define the area functional $E$ on $L^1(\Omega)$ by
\begin{equation*}
E(u_0)=\left\{
 \begin{alignedat}{1}
   \frac{1}{2}\int_{\Omega}|\nabla u_0| ~&~ \text{if} ~u_0\in BV (\Omega, \{1, -1\}), \\\
\infty~&~ \text{otherwise}.
 \end{alignedat} 
  \right.
  \end{equation*}
For a function of bounded variation $u_0\in BV (\Omega, \{1, -1\})$ taking values $\pm 1$, $|\nabla u_0|$ denotes the total variation of the vector-valued measure $\nabla u_0$ (see \cite{Simon}), and $\Gamma= \partial\{x\in \Omega: u_0(x)=1\}\cap \Omega$ the interface separating 
the phases of $u_0$. If $\Gamma$ is sufficiently regular (say $C^1$) then $E(u_0)=\mathcal{H}^{N-1}(\Gamma)$ and hence we identify $$E(u_0)\equiv E(\Gamma)=\mathcal{H}^{N-1}(\Gamma)$$ where $\mathcal{H}^{N-1}$ denotes the $(N-1)$-dimensional Hausdorff measure. In this paper, we are mostly concerned with $C^2$ interface $\Gamma$. Throughout, we denote by $\stackrel{\rightarrow}{n} = (n_{1},\cdots,n_{N})$ the outward unit normal to the region enclosed by $\Gamma$; and $(\cdot,\cdot)$ the standard inner product on $\RR^{N}$.

\subsection{Second inner variations of Allen-Cahn energies, defect measure and hidden 4-tensors}
In a previous paper \cite{Le}, we studied the relationship between the second inner variations of the Allen-Cahn functionals arising in the van der Waals-Cahn-Hilliard gradient theory of phase transitions \cite{AC}
\begin{equation} E_{\varepsilon}(u)=\int_{\Omega}\left(\frac{\varepsilon \abs{\nabla u}^2}{2} +\frac{(1-u^2)^2}{2\varepsilon}\right) dx ~(\e>0),\label{MM}\end{equation} 
where $u: \Omega\rightarrow \RR$ 
 and the second inner variation of their Gamma-limit which is the area functional 
\begin{equation*}E_2(u_0)\equiv \frac{4}{3} E(u_0)=\frac{2}{3}\int_{\Omega} |\nabla u_0| = \frac{4}{3} \mathcal{H}^{N-1}(\Gamma) := \frac{4}{3}E(\Gamma)\equiv E_2(\Gamma).\end{equation*} 
Contrary to the convergence of the first inner variations, we found in \cite{Le} a mysterious positive discrepancy term $\frac{4}{3}\int_{\Gamma} (\stackrel{\rightarrow}{n},\stackrel{\rightarrow}{n}\cdot\nabla\eta)^2d\mathcal{H}^{N-1}$ in the limit $\e\searrow 0$ of the difference of the second inner variations $\delta^{2} E_{\varepsilon}(u_{\varepsilon},\eta,\zeta)- \delta^{2}E_2(\Gamma,\eta,\zeta)$: If $u_{\e}\rightarrow u_0\in BV (\Omega, \{1, -1\})$ with a $C^2$ interface $\Gamma$ and $\lim_{\e\rightarrow 0} E_{\e}(u_{\e})= E_2 (u_0)\equiv E_2 (\Gamma)$ then for all smooth vector fields $\eta,\zeta\in (C_{c}^{1}(\Omega))^{N}$, we 
have 
\begin{equation*}
\lim_{\varepsilon\rightarrow 0}\delta^{2} E_{\varepsilon}(u_{\varepsilon},\eta,\zeta) = \frac{4}{3}\left\{
 \delta^{2}E(\Gamma,\eta,\zeta) + \int_{\Gamma} (\stackrel{\rightarrow}{n},\stackrel{\rightarrow}{n}\cdot\nabla\eta)^2d\mathcal{H}^{N-1}\right\}.\end{equation*} 

We view this discrepancy term as the {\it defect} measure of $\delta^{2} E_{\varepsilon}(u^{\varepsilon},\eta,\zeta)- \delta^{2}E_2(\Gamma,\eta,\zeta)$. This type of defect measure also appears in a related context. In \cite{RW}, R\"oger and Weber considered the stochastic Allen-Cahn equation
$$d u_{\e} = \left(\Delta u_{\e}-2\e^{-2} u_{\e}(u_{\e}^2-1)\right) dt + \nabla u_{\e}\cdot X(x, \circ dt)$$
where $X$ is a vector field valued Brownian motion. It is shown that at each time $t$, the defect measure between the localized energies associated with $u_{\e}(t)$ and the localized surface area of the sharp interface $\Gamma(t)$ of $u_{\e}(t)$ is of the form $\int_{\Gamma(t)}(\stackrel{\rightarrow}{n}, \stackrel{\rightarrow}{n}\cdot\nabla X^k)^2 d\mathcal{H}^{N-1}$
for suitable smooth time dependent vector fields $X^k$ on $\Omega$. 
It\^o formula is responsible for this extra term.\\

The purpose of this paper is to deterministically and conceptually explain the defect measure $\int_{\Gamma} (\stackrel{\rightarrow}{n},\stackrel{\rightarrow}{n}\cdot\nabla\eta)^2 d\mathcal{H}^{N-1}$, and reveal that this is a codimension-one phenomenon.
To do this, we imbed the usual Allen-Cahn functionals $E_{\e}$ into a family $E_{\e, p}$ of $p$-Laplace Allen-Cahn functionals that still Gamma-converge to the area functional. 
To be more precise, for $1\leq p<\infty$, let
$$E_{\e, p} (u)=\int_{\Omega} \left( \frac{\e^{p-1} |\nabla u|^p}{p} + \frac{(p-1)W(u)}{p\e}\right) dx,~~W(u) \equiv(1-u^2)^2.$$ 
Then, from the work of Bouchitt\'e \cite{Bou}, we know that $E_{\e, p}$ Gamma-converges to 
$$ E_{p}(\Gamma)= c_p \mathcal{H}^{n-1}(\Gamma)\equiv c_p E(\Gamma)~\text{with}~c_p:=\int_{-1}^{1} (W(s))^{\frac{p-1}{p}} ds.$$
In particular, the following conditions of Gamma-convergence hold:
\begin{myindentpar}{1cm}
1. ({\it Liminf inequality}) If $v_{\e_{i}}\rightarrow v_0$ in $L^1(\Omega)$ for some sequence $\e_{i}\rightarrow 0$ then
$$\liminf_{i\rightarrow \infty} E_{\e_{i}, p} (v_{\e_i})\geq E_p (v_0).$$
2. ({\it Existence of recovery sequence}) For any $w_0\in L^1(\Omega)$ there is a sequence $\{w_{\e_j}\}$ with $w_{\e_j}\rightarrow w_0$ in $L^1(\Omega)$ and 
$\lim_{j\rightarrow \infty} E_{\e_{j}, p} (w_{\e_j}) = E_p (w_0).$
\end{myindentpar}
Note that $E_{\e, 2} = E_{\e}$.
Observe that, when $p=1$ and $u\in BV (\Omega, \{1, -1\})$ with interface $\Gamma$, $E_{\e, p}(u)= E_1(u)\equiv 2E (\Gamma)= E_1(\Gamma).$ Thus, we expect that the second inner variations of $E_{\e,1}$ and $E_1$ are the same. This suggests that the extra term $\int_{\Gamma} (\stackrel{\rightarrow}{n},\stackrel{\rightarrow}{n}\cdot\nabla\eta)^2 d\mathcal{H}^{N-1}$ eventually disappears when passing the difference of the second inner variations $\delta^{2} E_{\varepsilon, p}(u_{\varepsilon},\eta,\zeta)- \delta^{2}E_p(\Gamma,\eta,\zeta)$ to the limits $\e\rightarrow 0$ first and then $p\searrow 1$. This is precisely what we prove here in our first main theorem.
\begin{theorem}
Fix $1<p<\infty$.
Let $u_{\e}$ be a sequence of functions that converges in $L^1(\Omega)$ to a function $u_0\in BV (\Omega, \{1, -1\})$ with a $C^2$ interface $\Gamma=\partial\{u_0=1\}\cap\Omega$. Assume that
 $\lim_{\varepsilon \rightarrow 0} E_{\varepsilon, p}(u_{\varepsilon}) = 
E_p(\Gamma).$
Then, for all smooth vector fields $\eta,\zeta\in (C_{c}^{1}(\Omega))^{N}$, we 
have
\begin{equation*}\lim_{\varepsilon\rightarrow 0}\delta^{2} E_{\varepsilon, p}(u_{\varepsilon}, \eta, \zeta) = c_p\left\{\delta^{2}E(\Gamma,\eta,\zeta) + (p-1)\int_{\Gamma} (\stackrel{\rightarrow}{n},\stackrel{\rightarrow}{n}\cdot\nabla\eta)^2 d\mathcal{H}^{N-1}\right\}.
\label{thm-ACp}
\end{equation*} 
\end{theorem}
 
In the above theorem, the second inner variation of $E$ at $\Gamma$ with respect to the velocity and acceleration vector fields $\eta$ and $\zeta$ in $(C_c^{1}(\Omega))^{N}$ is defined by 
(see \cite{Simon})
\begin{eqnarray}
 \delta^{2}E(\Gamma,\eta,\zeta)&: =& \left.\frac{d^2}{dt^2}\right\rvert_{t=0} \mathcal{H}^{N-1}(\Phi_t(\Gamma))\nonumber\\ &=& 
 \int_{\Gamma}\left\{ \text{div}^{\Gamma}\zeta + (\text{div}^{\Gamma}\eta)^2 + \sum_{i=1}^{N-1}\abs{(D_{\tau_{i}}\eta)^{\perp}}^2 -
 \sum_{i,j=1}^{N-1}(\tau_{i}\cdot D_{\tau_{j}}\eta)(\tau_{j}\cdot D_{\tau_{i}}\eta)\right\}d\mathcal{H}^{N-1}.
\label{SVE}
\end{eqnarray}
Here 
$\Phi_t$ is 
given by (\ref{map-deform}), $\text{div}^{\Gamma}\varphi$ denotes the tangential divergence of $\varphi$ on $\Gamma$; and for each point $x\in\Gamma$, $\{\tau_{1}(x),\cdots,\tau_{N-1}(x)\}$ is any orthonormal basis for the tangent space $T_{x}(\Gamma)$; for each $\tau\in T_{x}(\Gamma)$, $D_{\tau} \eta$ is the directional derivative and the normal part of $D_{\tau_{i}} \eta$ is denoted by 
 $(D_{\tau_{i}} \eta)^{\perp} = D_{\tau_{i}} \eta -\sum_{j=1}^{N-1}(\tau_{j}\cdot D_{\tau_{i}} \eta)\tau_{j}. $
\begin{remark}
The main results of this paper hold with $\Gamma$ having singular set of lower Hausdorff dimensions; precisely,
$\mathcal{H}^{N-3}(\text{sing}~\Gamma)=0.$
However, for the sake of clarity, we choose to present the main results with the assumption that $\Gamma$ is $C^2$ in $\Omega$.
\label{SingRm}
\end{remark}

Without entering into details of the proof, we explain here why Theorem \ref{thm-ACp} should be true.
By writing down the formula for $\delta^{2} E_{\varepsilon, p}(u_{\varepsilon}, \eta, \zeta)$, we see that all terms, except one, involve 2-tensor
$\e^{p-1} \nabla u_{\e}\otimes \nabla u_{\e}|\nabla u_{\e}|^{p-2}.$
The exception comes from the term involving 4-tensor
$\e^{p-1} \nabla u_{\e}\otimes \nabla u_{\e}\otimes\nabla u_{\e}\otimes \nabla u_\e|\nabla u_{\e}|^{p-4}.$
That is the term
$(p-2) \e^{p-1} (\nabla u_\e)^i (\nabla u_\e)^j |\nabla u_\e|^{p-4}(\nabla u_\e \cdot\nabla \eta)^i (\nabla u_\e \cdot\nabla \eta)^j$
which vanishes in the usual Allen-Cahn functionals $E_{\e}\equiv E_{\e, 2}$ and arises from
\begin{multline*}F_{p_i p_j}(u_\e, \nabla u_\e) (\nabla u_\e \cdot\nabla \eta)^i (\nabla u_\e \cdot\nabla \eta)^j=
\e^{p-1}\delta_{ij}(\nabla u_\e \cdot\nabla \eta)^i (\nabla u_\e \cdot\nabla \eta)^j |\nabla u_\e|^{p-2} \\+ 
(p-2) \e^{p-1} (\nabla u_\e)^i (\nabla u_\e)^j |\nabla u_\e|^{p-4}(\nabla u_\e \cdot\nabla \eta)^i (\nabla u_\e \cdot\nabla \eta)^j,
\end{multline*}
 where
$$F(z, {\bf p}) = \frac{\e^{p-1} |{\bf p}|^p}{p} + \frac{(p-1)W(z)}{p\e}~; {\bf p}=(p_1, \cdots, p_N).$$
The important fact, proved in Lemma \ref{propo-Res}, is the convergence of Reshetnyak type of the following 2-tensors
\begin{equation}\e^{p-1} \nabla u_{\e}\otimes \nabla u_{\e} |\nabla u_{\e}|^{p-2}\rightharpoonup c_p \stackrel{\rightarrow}{n}\otimes \stackrel{\rightarrow}{n}\mathcal{H}^{N-1}\lfloor \Gamma
\label{2nu0}
\end{equation}
and 4-tensors
\begin{equation}\e^{p-1} \nabla u_{\e}\otimes \nabla u_{\e}\otimes\nabla u_{\e}\otimes \nabla u_{\e}|\nabla u_{\e}|^{p-4}\rightharpoonup c_p \stackrel{\rightarrow}{n}\otimes \stackrel{\rightarrow}{n}\otimes \stackrel{\rightarrow}{n} \otimes \stackrel{\rightarrow}{n}\mathcal{H}^{N-1}\lfloor \Gamma.
\label{4nu0}
\end{equation}

Using Reshetnyak type convergence result and passing to the limit in the second inner variations $\delta^{2} E_{\varepsilon, p}(u_{\varepsilon}, \eta, \zeta)$, we can easily write schematically
$$\lim_{\varepsilon\rightarrow 0}\frac{1}{c_p}\delta^{2} E_{\varepsilon, p}(u_{\varepsilon}, \eta, \zeta)= \lim_{\varepsilon\rightarrow 0}\frac{1}{c_2}\delta^{2} E_{\varepsilon, 2}(u_{\varepsilon}, \eta, \zeta) + (p-2)\int_{\Gamma} (\stackrel{\rightarrow}{n},\stackrel{\rightarrow}{n}\cdot\nabla\eta)^2 d\mathcal{H}^{N-1}. $$ 
Thus, if we write
$$\lim_{\varepsilon\rightarrow 0}\frac{1}{c_2}\delta^{2} E_{\varepsilon, 2}(u_{\varepsilon}, \eta, \zeta) = \delta^{2}E(\Gamma,\eta,\zeta) + \text{possible extra term}$$
then we have
$$\lim_{\varepsilon\rightarrow 0}\frac{1}{c_p}\delta^{2} E_{\varepsilon, p}(u_{\varepsilon}, \eta, \zeta)= \delta^{2}E(\Gamma,\eta,\zeta) + \text{possible extra term} + (p-2)\int_{\Gamma} (\stackrel{\rightarrow}{n},\stackrel{\rightarrow}{n}\cdot\nabla\eta)^2 d\mathcal{H}^{N-1}. $$ 
By letting $p\searrow 1$ and noting that the left hand side is expected to be $\delta^{2}E(\Gamma,\eta,\zeta)$, we obtain that 
$$\text{possible extra term} = \int_{\Gamma} (\stackrel{\rightarrow}{n},\stackrel{\rightarrow}{n}\cdot\nabla\eta)^2 d\mathcal{H}^{N-1}.$$
This explains the mysterious extra term in Theorem \ref{thm-ACp} and why Theorem \ref{thm-ACp} should be true. Our analysis reveals that the convergence of 4-tensors in (\ref{4nu0}) is responsible for the appearance of the extra term when $p>1$ and its disappearance in the limit $p\searrow 1$. These 4-tensors are hidden in the usual Allen-Cahn functional $E_{\e}$. 

We now turn to the case where $u_{\e}$ is complex-valued and satisfying similar assumptions as in Theorem \ref{thm-ACp}, as in the case of critical points of the 
Ginzburg-Landau functional in superconductivity. We still have a convergence of 2-tensors as in (\ref{2nu0}) (see (\ref{resGL})) while for 4-tensors, (\ref{4nu0}) does not seem to hold 
anymore. Therefore, the discrepancy term in the limit of $\delta^{2} E_{\varepsilon, p}(u_{\varepsilon},\eta,\zeta)- \delta^{2}E_p(\Gamma,\eta,\zeta)$ for Ginzburg-Landau is 
expected to be of different nature than in the case of the Allen-Cahn functionals. We find an alternative formula for this discrepancy term which, interestingly, involves Jacobian determinant and the $\bar{\partial}$-operator. As a consequence, we show  that the stability and instability of Ginzburg-Landau vortices in higher dimensions pass to the limit provided that the limiting vortex filament is smooth and connected. All of these will be made precise in Section \ref{sec-GL}. \\

As an application of Theorem \ref{thm-ACp}, we partially answer a question of Kohn and Sternberg \cite{KS} by giving a relation between the limit of second variations of the Allen-Cahn functional and the second variation of the area functional at local minimizers. This is the subject of the next section.
\subsection{Local minimizers of Allen-Cahn type functionals} 
In \cite{KS}, Kohn and Sternberg discovered a very interesting connection between isolated local minimizers of the area functional and the existence of local minimizers of $E_\e$. They proved the following theorem.
\begin{theorem}(\cite[Theorem 2.1]{KS})
Let $\Omega$ be a bounded domain in $\RR^N$ with Lipschitz boundary, and suppose that $u_0$ is an isolated $L^1$-local minimizer of $E.$ Then there exists $\e_0>0$ and a family $\{u_{\e}\}_{\e<\e_0}$ such that
\begin{center}
$u_\e$ is an $L^1$-local minimizer of $E_\e$, and
$\|u_\e-u_0\|_{L^1(\Omega)}\rightarrow 0$ as $\e\rightarrow 0.$
\end{center}
\label{isolate1}
\end{theorem}
We recall relevant concepts here. We call $u_0$ an isolated $L^1$-local minimizer of $E$ if
$$E(u_0)< E(u)~\text{whenever}~0<\|u-u_0\|_{L^1(\Omega)}\leq\delta$$
for some $\delta>0$. Similarly, we call $u_\e$ an $L^1$-local minimizer of $E_\e$ if for some $\delta>0$, we have
$$E_\e(u_\e)\leq E_{\e}(u)~\text{whenever}~\|u_\e-u\|_{L^1(\Omega)}\leq\delta.$$
It is still an open question whether $u_\e$ is isolated. Kohn and Sternberg also asked \cite[Remark 2. 3]{KS} if there is any connection between the second variation of $E_\e$ at $u_{\e}$ given by
$$d^2 E_{\e}(u_\e, \varphi)= \int_{\Omega} [\e |\nabla \varphi|^2 + 2\e^{-1} (3u_\e^2-1)\varphi^2] dx$$
and the second (inner) variation of $E$ at $\Gamma$. \\
\h In this paper, we partially answer the above question of Kohn and Sternberg by providing a relationship between the second variation of $E_{\e, p}$ and the second inner variation of $E$ at local minimizers in the more general setting of $p$-Laplace Allen-Cahn energies. This is the content of the following theorem.
\begin{theorem} 
\label{isolate2}Fix $p\in (1, \infty)$. Let $\Omega$ be a bounded domain in $\RR^N$ with Lipschitz boundary, and suppose that $u_0$ is an isolated $L^1$-local minimizer of $E$ with a $C^2$ interface $\Gamma=\partial\{u_0=1\}\cap\Omega$. Then there exists $\e_0>0$ and a family $\{u_{\e, p}\}_{\e<\e_0}$ such that
\begin{center}
$u_{\e, p}$ is an $L^1$-local minimizer of $E_{\e. p}$, and
$\|u_{\e, p}-u_0\|_{L^1(\Omega)}\rightarrow 0$ as $\e\rightarrow 0.$
\end{center} 
With these $u_{\e, p}$, for all smooth vector fields $\eta,\zeta\in (C_{c}^{1}(\Omega))^{N}$, we 
have \begin{equation}\lim_{\varepsilon\rightarrow 0}d^{2} E_{\varepsilon, p}(u_{\varepsilon, p}, -\nabla u_{\e, p}\cdot\eta) = c_p\left\{
 \delta^{2}E(\Gamma,\eta,\zeta) +  (p-1)\int_{\Gamma} (\stackrel{\rightarrow}{n},\stackrel{\rightarrow}{n}\cdot\nabla\eta)^2 d\mathcal{H}^{N-1}\right\}.\label{discrep}\end{equation} 
\end{theorem}
We recall that, if $\Gamma$ is an isolated $L^1$-local minimizer of the area 
functional, then its smoothness is guaranteed in dimensions $N\leq 7$ while its singular set has Hausdorff dimension at most $N-8$ in dimensions $N>7$; see \cite[Theorem 37.7]{Simon}.
Thus, by Remark \ref{SingRm}, the result of Theorem \ref{isolate2} hold for all dimensions $N\geq 2$ without the 
assumption that $\Gamma$ is $C^2$ in $\Omega$. It is worth noting that, by recent work of Tonegawa and Wickramasekera \cite{TW}, the above result on the smoothness and/or singularity of
$\Gamma$ still holds when $\Gamma$ is the limiting interface of a sequence of stable solutions of 
the Allen-Cahn equation.\\

In the special case $p=2$, Theorem \ref{isolate2} gives the upper semicontinuity of the eigenvalues of the operators $-\e \Delta + 2\e^{-1}(3u_{\e}^2-1)$ in the limit. The precise statement is as follows. 
\begin{cor} Assume that $p=2$. Let $u_0$ and $u_{\e}=u_{\e, 2}$ be as in Theorem \ref{isolate2}.  Assume that $\Gamma$ is connected. Let $\lambda_{\e, k}$ be the $k$-th eigenvalue of the operator $-\e \Delta + 2\e^{-1}(3u_{\e}^2-1)$ in $\Omega$ with zero Dirichlet condition on $\partial\Omega$. Let $\lambda_k$ be the $k$-th eigenvalue of the operator $-\Delta_{\Gamma} + |A|^2$ in $\Gamma$ with zero Dirichlet condition on $\partial\Gamma$. Then
$$\limsup_{\e\rightarrow 0}\frac{\lambda_{\e, k}}{\e}\leq \lambda_k.$$
\label{eigen_cor}
\end{cor}

The method of the proof of Theorem \ref{isolate2} answering a question of Kohn and Sternberg can be appreciated more when a volume constraint is present in the Allen-Cahn functional $E_{\e}$ and the area functional $E$. This is the subject of the next section.
\subsection{The second inner variations of Allen-Cahn type energies with volume constraint}

The purpose of this section is to prove an analog of Theorem \ref{isolate2} for isolated local minimizers of the area functional with volume constraint, say 
\begin{equation}\bar{u}_{\Omega}:=\frac{1}{\abs{\Omega}}\int_{\Omega} u(x) dx = m\in (-1,1).
\label{volc} 
\end{equation}
Suppose that $u_0\in BV (\Omega, \{1, -1\})$ with interface $\Gamma=\partial E_0\cap\Omega$ 
where
$$E_0=\{x\in\Omega: u_0(x)=1\}.$$ 

When $E_0$ is stable for the area functional $E$ with volume constraint (\ref{volc}) for $u=u_0$, Sternberg and Zumbrun \cite{SZ2} derived the following Poincar\'e inequality
\begin{equation}J(\xi):=\int_{\Gamma} \left(|\nabla_{\Gamma}\xi|^2 -|A_{\Gamma}|^2|\xi|^2\right) d\mathcal{H}^{N-1} - \int_{\partial\Gamma\cap\partial\Omega} A_{\partial\Omega}(\stackrel{\rightarrow}{n}, \stackrel{\rightarrow}{n})|\xi|^2 d\mathcal{H}^{N-2}\geq 0
\label{PI1}
\end{equation}
for all smooth functions $\xi$ satisfying
$\displaystyle\int_{\Gamma} \xi(x) d\mathcal{H}^{n-1}(x)=0.$ Here we used the notation $A_M$ to denote the second fundamental form of the manifold $M$.

We recall here relevant concepts from \cite{SZ2}. A family of subsets of $\Omega$ which are deformations of $E_0$, $\{E_t\}_{t\in (-T, T)}$ for some $T>0$,  is called {\bf admissible} if
\begin{center}
$\chi_{E_t}\rightarrow \chi_{E_0}$ in $L^1(\Omega)$ as $t\rightarrow 0$,~
$t\rightarrow \mathcal{H}^{N-1}(\partial E_t\cap\Omega)$ is twice differentiable at $t=0$,
and $|E_t|= |E_0| + o(t^2)$.
\end{center}
\begin{definition}
We will call $E_0$ stationary for the area functional $E$ with volume constraint (\ref{volc}) if $\left.\frac{d}{dt}\right\rvert_{t=0}\mathcal{H}^{N-1}(\partial E_t\cap\Omega)=0$ for all admissible families $\{E_t\}$. We will call $E_0$ stable if $E_0$ is stationary and $\left.\frac{d^2}{dt^2}\right\rvert_{t=0}\mathcal{H}^{N-1}(\partial E_t\cap\Omega)\geq 0$ for  all admissible families $\{E_t\}$.
\end{definition}
As in the calculation (\ref{SVE}) which also holds for vector fields compactly supported in $\RR^N$ \cite{Simon}, given $\eta, \zeta\in (C^1(\overline{\Omega}))^{N}$, we extend them to be compactly supported vector fields in $\RR^N$.  For the purpose of 
calculating the second inner variation $\delta^2 E(\Gamma, \eta, \zeta)$ of $E$ with volume constraint, we have the following definition which is motivated by (\ref{SVE}).
\begin{definition} A family $\tilde E_{t}=\Phi_t(E_0)$ of deformations of $E_0$ where $\Phi_t$ is defined by (\ref{map-deform}) is called {\bf domain admissible} if this family preserves the volume of $E_0$ up to second order in $t$, that is, 
$|\tilde E_t|= |E_0| + o(t^2).$
\label{t2def}
\end{definition}
Note that, while $\Phi_t$ in (\ref{map-deform}) is primarily defined for vector fields $\eta$ and $\zeta$ compactly supported in $\Omega$, it is not the case here in Definition \ref{t2def}.
In general, $\tilde E_t$ in Definition \ref{t2def} can go outside of $\Omega$.

The next theorem reveals the connection between Poincar\'e inequality and the second inner variation for functionals with volume constraint.
\begin{theorem}
With the notations as above,
\begin{myindentpar}{1cm}
(i)The family $\displaystyle\{\tilde E_t=\Phi_t(E_0)\}$ is domain admissible only if $\int_{E_0} \text{div}~ \eta~ dx=0.$ Vice versa, 
whenever $\eta$ satisfies $\int_{E_0} \text{div}~ \eta~ dx=0$, one can choose $\zeta=- (div\eta)\eta + (\eta\cdot\nabla)\eta$ so that the family $\displaystyle \tilde E_t$ becomes domain admissible.
If $\tilde E_t$ is domain admissible then a particular second inner variation of $E(\Gamma)$ with volume constraint (\ref{volc}) and velocity $\eta$ is 
$\delta^2 E (\Gamma, \eta, \zeta^{\eta})$ where
$\zeta^{\eta}:= - (div\eta)\eta + (\eta\cdot\nabla)\eta.$ In this formula, 
the $i$-th component of $(\eta\cdot\nabla)\eta$ is 
$\displaystyle\sum_{j}\frac{\partial\eta^i}{\partial x_j}\eta^j.$\\
(ii) In the special case where $E_0$ is stationary for the area functional $E$ with volume constraint (\ref{volc}), $\eta$ is a smooth vector field  tangent to $\partial\Omega$, normal to $\Gamma$ with $(\stackrel{\rightarrow}{n}, \stackrel{\rightarrow}{n}\cdot\nabla \eta)=0$ on $\Gamma$ and
$\displaystyle\int_{\Gamma} \eta(x)\cdot\stackrel{\rightarrow}{n}(x) d\mathcal{H}^{N-1}(x)=0,$
we have
$$\delta^2 E (\Gamma, \eta, \zeta^{\eta})= J(\eta\cdot \stackrel{\rightarrow}{n}).$$
Moreover, if $E_0$ is stable 
then the Poincar\'e inequality
$J(\eta\cdot \stackrel{\rightarrow}{n})\geq 0$ holds.\\
(iii) Let $u_{\e}$ and $u_{0}$ be as in Theorem \ref{isolate2} but now equipped with the volume constraint (\ref{volc}). Let $\eta\in (C_{c}^{2}(\Omega))^{N}$ be such that $\displaystyle\int_{\Gamma} \eta \cdot \stackrel{\rightarrow}{n}\mathcal{H}^{N-1}=0$. Then for any $C^2$ perturbation vector field $\eta^{\e}\in (C_{c}^{2}(\Omega))^{N}$ of $\eta$ satisfying 
$$\lim_{\e\rightarrow 0}\|\eta^{\e}-\eta\|_{C^{2}(\overline{\Omega})}=0,~\int_{\Omega} \nabla u_{\e}\cdot\eta^{\e} dx=0,$$ we have
$$\lim_{\e\rightarrow 0}d^2 E_{\e} (u_{\e},  -\nabla u_{\e}\cdot\eta^{\e}) = c_2\left\{
 \delta^{2}E(\Gamma,\eta,\zeta^{\eta}) +  \int_{\Gamma} (\stackrel{\rightarrow}{n},\stackrel{\rightarrow}{n}\cdot\nabla\eta)^2 d\mathcal{H}^{N-1}\right\}.$$

\end{myindentpar}
\label{isolate3}
\end{theorem}

The Poincar\'e inequality (\ref{PI1}) was later extended by Choksi and Sternberg \cite{CS} to the nonlocal area functional with a long-range interaction modeling diblock-copolymers. Theorem \ref{isolate3} can also be extended to this nonlocal setting. 
It is now worth commenting briefly on the method of the proof of (\ref{PI1}) in \cite{SZ2} (see also \cite{CS}) and our approach using the second inner variation. \\

The idea in \cite{SZ2} is to apply the stability inequality for an admissible family $\{E_t\}$ of deformations of $E_0$ using the diffeomorphism $\Psi_t$ generated by the vector 
field $\eta$ satisfying the assumptions of Theorem \ref{isolate3} (ii), that is, $E_t=\Psi_t (E_0)$ where $\Psi_t$ is the solution to
\begin{equation}\frac{\partial \Psi_t(x)}{\partial t} =\eta (\Psi_t(x)),~ \Psi_0(x) =x~\text{for all~}x\in \RR^N.
\label{SZpsi}
\end{equation}
The volume of $\Psi_t(E_0)$ is preserved up to first order but, in general, not up to second order in $t$. Thus a second order modification is needed. 

In our approach using second inner variation, we can produce domain admissible family $\{\tilde E_t\}$ and admissible family $\{E_t\}$ at the same time. They are the same if $\eta$ and $\zeta$ are compactly supported in $\Omega$. Moreover, the second order modification is already {\it built in} the acceleration vector $\zeta$. Any modification, if necessary, corresponds to a choice of $\zeta$. In the problem at hand with a volume constraint (\ref{volc}), what is needed is that the following identity
$$div \zeta + (div\eta)^2 - trace((\nabla\eta)^2)=0.$$
It is a remarkable, yet very simple, fact that the difference of the last two nonlinear terms in the above equation is a divergence of a vector field. In fact, we have
\begin{equation}(div\eta)^2 - trace((\nabla\eta)^2)= div \{(div\eta)\eta- (\eta\cdot\nabla)\eta\}.
\label{GoodIden}
\end{equation}
This is why we choose $\zeta$ to be $\zeta^{\eta}$ in the Theorem \ref{isolate3}. This explicit expression for $\zeta$ is the key in the proof of Parts (ii)-(iii) of Theorem \ref{isolate3}. 
\begin{remark}
The identity (\ref{GoodIden}) was used crucially by Lin \cite{Lin} in his elegant proof of the minimality property of the harmonic map $\frac{x}{|x|}: B^n \rightarrow S^{n-1}$ ($n\geq 3$) among all maps $\varphi: B^n\rightarrow S^{n-1}$ with $\varphi =x$ on $S^{n-1}$. His proof consists of proving that
\begin{equation*}
|\nabla \varphi|^2 \geq \frac{1}{n-2} \left((div\varphi)^2 - trace((\nabla\varphi)^2 \right)~\text{for}~\varphi: B^n\rightarrow S^{n-1}
\end{equation*}
and then integrating the right hand side using (\ref{GoodIden}).
\end{remark}

\begin{remark}

For $u_{\e} + t\varphi$ to be a variation of $u_{\e}$ for the purpose of calculating the second variation $d^2 E_{\e}(u_{\e},\varphi)$ under
the volume constraint (\ref{volc}), $\varphi$ must satisfy $\int_{\Omega}\varphi~ dx=0$. In general, $\int_{\Omega} \nabla u_{\e}\cdot\eta~ dx \neq 0$ for $\eta$ in part (iii).  Therefore, we must need $C^2$ perturbations $\eta^{\e}$ of $\eta$ so that $\int_{\Omega} \nabla u_{\e}\cdot\eta^{\e}~ dx =0$ in order to calculate $d^2 E_{\e}(u_{\e},-\nabla u_{\e}\cdot\eta^{\e})$. Here is a simple way to construct $\eta^{\e}$ (see also \cite[Lemma 8.1]{Le_SIMA}). By the divergence theorem, it suffices to have
\begin{equation}\int_{\Omega} u_{\e} div \eta^{\e}~ dx=0.
\label{FVe}
\end{equation}
Choose any smooth vector field $\varphi\in (C_{c}^{2}(\Omega))^{N}$ satisfying
$\int_{\Gamma}\varphi\cdot \stackrel{\rightarrow}{n}\neq 0.$
Let
$$h(\e):=\frac{-\int_{\Omega} u_{\e}div \eta~dx}{\int_{\Omega} u_{\e}div \varphi~dx}~\text{and}~\eta^{\e}= \eta (x) + h(\e) \varphi(x).$$
Then, (\ref{FVe}) is satisfied and as $\e\rightarrow 0,$ we have
$$h(\e)\rightarrow \frac{-2\int_{E_0}div\eta~dx}{2\int_{E_0} div\varphi~dx}= \frac{-2\int_{\Gamma}\eta\cdot \stackrel{\rightarrow}{n}d\mathcal{H}^{N-1}}{2\int_{\Gamma}\varphi\cdot \stackrel{\rightarrow}{n}d\mathcal{H}^{N-1}}=0.$$
\label{eremark}
\end{remark}

\subsection{The second inner variations of Ginzburg-Landau energies }
\label{sec-GL}
Let $\Omega$ be an open smooth bounded set in $\RR^{N}$ ($N\geq 3$).
Consider the Ginzburg-Landau equation for $0<\varepsilon<1$
\begin{equation}
- \Delta u_{\varepsilon} =\frac{1}{\varepsilon^2} u_{\varepsilon} (1-\abs{u_{\varepsilon}}^2)~\text{in} ~\Omega,~
 u_{\varepsilon}  = g_{\varepsilon}~\text{on}~\partial\Omega.
\label{GL}
\end{equation}
Here $u_{\e}: \Omega\rightarrow C$ and $g_{\e}: \partial\Omega\rightarrow C$ are complex-valued functions.
A solution $u_{\e}$ of (\ref{GL}) is a critical point of the simplified Ginzburg-Landau energy in superconductivity which is a complex analog of (\ref{MM}):
\begin{equation*}
 E_{\varepsilon}(u) = \frac{1}{ \abs{log\e}}\int_{\Omega} \left(\frac{1}{2}\abs{\nabla u}^2 + \frac{1}{4\varepsilon^2} (1-\abs{u}^2)^2 \right)dx
\equiv \int_{\Omega} \frac{e_{\e} (u)}{|log\e|}dx.
\end{equation*}
We assume that the energy of $u_{\e}$ satisfies
\begin{equation}
 E_{\varepsilon}(u_{\varepsilon}) \leq K.
\label{energybound}
\end{equation}
The existence of solutions of (\ref{GL}) satisfying (\ref{energybound}) can be proved for very general $g_{\e}$ allowing singularities of dimension $N-3$ on $\partial\Omega$ (see \cite[Condition (H2)]{BBO}).
With (\ref{energybound}), we have
\begin{equation*}
 e_{\e}(u_{\e})/|log\e| \rightharpoonup \mu_{\ast}~\text{ in the sense of Radon measures}~ \text{where}~ \mu_{\ast}~ \text{is a bounded measure on}~ \overline{\Omega}.  
\end{equation*}
Properties of $\mu_{\ast}$ can be found in \cite[Theorem 1]{BBO}: $\mu_{\ast}$ is a stationary varifold;  the support $\Gamma$of $\mu_{\ast}$ is a closed subset of $\overline{\Omega}$ and $\mathcal{H}^{N-2}$-rectifiable with $E(\Gamma):=\mathcal{H}^{N-2}(\Gamma)<\infty$. $\Gamma$ is often refereed to as the limiting filament since it is the limit of zero set of $u_{\e}$. 

In this paper,  we consider a model case where we assume that $\Gamma$ is smooth and connected. Thus $\Gamma$ is a minimal submanifold.  An interesting question is then: 
\begin{center}\it If $u_{\e}$
are stable solutions to (\ref{GL}), is $\Gamma$ a stable submanifold?
\end{center} This question was answered in the affirmative by Serfaty in the two dimensional case \cite{Serfaty}. Here, we address the above question in the higher dimensional case via the second inner variation as in the case of Allen-Cahn functional \cite{Le}.
The main task now is to calculate the second inner variation of $E_{\e}$ and then pass to the limit. From the discussion following Theorem \ref{thm-ACp},  we do not expect to get a similar ``discrepancy formula''
as in Theorem \ref{thm-ACp}. An alternative formula is given in the following.
\begin{theorem} With the above assumptions, we can find a positive constant $m$  such that
$\lim_{\varepsilon\rightarrow 0}E_{\varepsilon}(u_{\varepsilon}) =
m\pi E(\Gamma)\equiv m\pi \mathcal{H}^{N-2}(\Gamma)$ while for all smooth vector fields $\eta,\zeta\in (C_{c}^{1}(\Omega))^{N}$, we have
\begin{equation}\lim_{\varepsilon\rightarrow 0}\delta^{2} E_{\varepsilon}(u_{\varepsilon},\eta,\zeta) =
m\pi\delta^{2}E(\Gamma,\eta,\zeta) + m\pi\int_{\Gamma} \left(\abs{D_{\perp} (\eta^{\perp})}^2 - 2Jac_{\perp} (\eta^{\perp})\right)d\mathcal{H}^{N-2}.\label{discrep0}\end{equation} 
Here $
 \eta^{\perp} =\eta -\sum_{i=1}^{N-2}(\eta,\tau_{i})\tau_{i}$
and $D_{\perp}$ and $Jac_{\perp}$ are the derivative and Jacobian taken in the orthogonal plane to $\Gamma$. 

As a consequence of the above formula, stability and instability of Ginzburg-Landau in higher dimensions also pass to the limit provided that the limiting vortex filament is smooth and connected.
 
\label{mainSV}
\end{theorem}
In the above theorem, we denote $\{\tau_{1}(x),\cdots,\tau_{N-2}(x)\}$ any orthonormal basis for the tangent space $T_{x}(\Gamma)$ for each $x\in\Gamma$.  

It is interesting to note that the Jacobian determinant appears in the above formula which is very natural in the Ginzburg-Landau setting. We can also write the discrepancy 
term using the $\bar{\partial}$-operator as follows. Suppose that the tangent space $T_{x}\Gamma$ is spanned by the standard unit vectors $\{e_{1}, \cdots, e_{N-2}\}$. We complexify the normal space $(T_x\Gamma)^{\perp}$using the complex variable $z_{\Gamma}= x_{N-1} + i x_N$. Then, we complexify the components of $\eta^{\perp}= (0, \cdots, 0, \eta^{N-1},\eta^{N})$ into a complex function $(\eta^{\perp})^{C}= \eta^{N-1} + i \eta^{N}$. Denote by $\overline{z}$ the complex conjugate of $z$. Then, we recall that for complex-valued $f$ defined on $(T_x\Gamma)^{\perp}$, we have
$$\frac{\partial f}{\partial \overline{z_{\Gamma}}}= \frac{1}{2}(\frac{\partial f}{\partial x_{N-1}} + i\frac{\partial f}{\partial x_{N}} ).$$
Now, a little computation shows that (see the end of the proof of Theorem \ref{mainSV})
$$\abs{D_{\perp} (\eta^{\perp})}^2 - 2Jac_{\perp} (\eta^{\perp}) =(\frac{\partial\eta^{N-1}}{\partial x_{N}}
+\frac{\partial\eta^{N}}{\partial x_{N-1}})^2 + (\frac{\partial\eta^{N-1}}{\partial x_{N-1}}
-\frac{\partial\eta^{N}}{\partial x_{N}})^2  = 4|\frac{\partial (\eta^{\perp})^{C}}{\partial \overline{z_{\Gamma}}}|^2.$$
Therefore, (\ref{discrep0}) becomes
\begin{equation}\lim_{\varepsilon\rightarrow 0}\delta^{2} E_{\varepsilon}(u_{\varepsilon},\eta,\zeta) =m\pi
\delta^{2}E(\Gamma,\eta,\zeta) + 4m\pi\int_{\Gamma}|\frac{\partial (\eta^{\perp})^{C}}{\partial \overline{z_{\Gamma}}}|^2 d\mathcal{H}^{N-2}.\label{discrep2}\end{equation}
\begin{remark}
With this expression, we discover that, for a vector field $\eta$ defined initially on $\Gamma$,   its holomorphic extension into the orthogonal plane of $\Gamma$ will make the discrepancy term vanish. 
\end{remark}

\begin{remark}
In \cite{MSZ}, Montero-Sternberg-Ziemer considered certain bounded, open, Lipschitz domain
$\Omega\subset \RR^3$ containing a collection of line segments $l_1, \cdots, l_N$ with some specific properties. 
Let
$\Gamma=\bigcup_{j=1}^{N} l_j.$
Then,  the authors constructed in \cite[Proposition 3.1 and Theorem 4.2]{MSZ} local minimizers $u_{\e}\in W^{1,2}(\Omega; C)$ in $W^{1,2}(\Omega; C)$ of $E_{\e}$ such that
$$\lim_{\e\rightarrow 0} E_{\e}(u_{\e})=\pi \mathcal{H}^{1}(\Gamma)~
\text{and}~
e_{\e}(u_{\e})/|log\e|\rightharpoonup \pi\mathcal{H}^{1}\lfloor \Gamma.$$
For these $u_{\e}$, we can use (\ref{SV_eq}) and Theorem \ref{mainSV} with $m=1$ to obtain as in Theorem \ref{isolate2}
$$\lim_{\varepsilon\rightarrow 0}d^{2} E_{\varepsilon}(u_{\varepsilon}, -\nabla u_\e\cdot\eta) =\pi\delta^{2}E(\Gamma,\eta,\zeta) + 4\pi\int_{\Gamma}|\frac{\partial (\eta^{\perp})^{C}}{\partial \overline{z_{\Gamma}}}|^2 d\mathcal{H}^{1}.$$
This is the relation between the second variation of $E_\e$ and that of $E$.
\end{remark}
\subsection{Further questions}
We list here some questions for further investigation.

\begin{myindentpar}{1cm}
{\bf 1. Ginzburg-Landau energies and codimension two area functional.}  Can we prove similar results as in Theorems \ref{thm-ACp} and \ref{isolate2} for Ginzburg-Landau energies? \\
{\bf 2. The higher dimensional area functional.} Essentially, we do not know any formula like those in Theorem \ref{thm-ACp} for the higher dimensional area functional and its variational approximation (see \cite{ABO}). This question is almost unexplored.  
\end{myindentpar}
The paper is organized as follows. In Section \ref{sec-Rel}, we establish a relationship between two notions of variations. We use this relationship to prove Theorem \ref{isolate2} assuming Theorem \ref{thm-ACp} and then Theorem \ref{isolate3}. We prove Theorem \ref{thm-ACp} in Section \ref{sec-ACp}. The proof of Theorem \ref{mainSV} will be given in Section \ref{sec-GLSV}. 

\section{A relation between two notions of variation and application to local minimizers}
\label{sec-Rel}
In this section, we prove Theorem \ref{isolate2} assuming Theorem \ref{thm-ACp}, Corollary \ref{eigen_cor} and then Theorem \ref{isolate3}. To do these, we use a relationship between two notions of variation stated in the following.
\begin{propo}
Up to second order, the inner variations of the functional $A$, defined in the Introduction, at $u$ with respect to smooth, compactly supported vector fields $(\eta, \zeta)$ are equal to the variations of $A$ at $u\in C^2(\Omega)$ with respect to $-\nabla u\cdot \eta$. More precisely, we have
\begin{equation}
\label{FV_eq} \delta A (u, \eta, \zeta)= dA (u, -\nabla u\cdot\eta)
\end{equation}
and  
\begin{equation}
\label{SV_eq} \delta^2 A (u, \eta, \zeta)= d^2A (u, -\nabla u\cdot\eta) + dA(u, X_0)
\end{equation}
where 
\begin{equation}X_0= (D^2 u(y) \cdot \eta(y), \eta (y)) + (\nabla u(y), 2\nabla \eta (y) \eta (y)-\zeta(y)).
\label{Xzero}
\end{equation}
\label{all_var}
\end{propo}
\begin{remark} \begin{myindentpar}{1cm}
(1) The identity (\ref{FV_eq})
is the main reason why we should multiply $\nabla u\cdot\eta$ to the Euler-Lagrange equation/chemical potential in phase transitions in order to obtain Gibbs-Thomson law/monotonicity formula. The idea is to go from the first variations to the first inner variations where we can pass to the limit (to obtain the corresponding first inner variations of the area functional). The most relevant works related to the subject of this paper are those of Luckhaus-Modica \cite{LM} and Tonegawa \cite{Tone2002}. \\
(2) For critical points of Allen-Cahn type energies such as those of (\ref{MM}) and (\ref{GL}), formula (\ref{SV_eq}) is already known in the literature \cite{Le, Serfaty}; its proof can be seen by direct calculations using the Euler-Lagrange equation. Our formula (\ref{SV_eq}) generalizes the above mentioned formula in \cite{Le, Serfaty}. It holds for general $u$, not necessarily critical points of $A$,  and of independent interest. It is especially relevant when the first variation of $A$ does not vanish as in the case of critical points with constraints in Theorem \ref{isolate3}.
\end{myindentpar}
\end{remark}

Now, we are ready to prove Theorem \ref{isolate2}.
\begin{proof}[Proof of Theorem \ref{isolate2}] If $\Gamma$ is an isolated $L^1$- local minimizer of the area functional $E$ then so is for $E_p.$ 
The construction of $u_{\e, p}$ and the proof of Theorem \ref{isolate1} in \cite{KS} give a sequence of $L^1$-local minimizers $u_{\e, p}$ of $E_{\e, p}$ such that
$$\|u_{\e, p}-u_0\|_{L^1(\Omega)}\rightarrow 0~ \text{as} ~\e\rightarrow 0~\text{and}~\lim_{\e\rightarrow 0} E_{\e, p}(u_{\e, p}) = E_{p}(\Gamma).$$
For completeness, we sketch the proof. Since $u_0$ is isolated, we can choose $\delta>0$ such that
\begin{equation}E_p(u_0)< E_p(u)~\text{whenever}~0<\|u-u_0\|_{L^1(\Omega)}\leq\delta.
\label{isodef}
\end{equation}
Let $u_{\e, p}$ be any minimizer of $E_{\e, p}$ on the ball
$$B=\{u: \|u-u_0\|_{L^1(\Omega)}\leq \delta\}.$$
The existence of such a $u_{\e, p}$ is guaranteed by the direct method of the calculus of variations. Since $E_{\e, p}$ Gamma-converges to $E_p$, there is a sequence $\{w_{\e_i, p}\}$ with $w_{\e_i, p}\rightarrow u_0$ in $L^{1}(\Omega)$ and $E_{\e_i, p}(w_{\e_i, p})\rightarrow E_{p}(u_0)$. When $\e_i$ is small, $w_{\e_i, p}$ lies in $B$. It follows that
$$\liminf E_{\e, p} (u_{\e, p})\leq E_p(u_0).$$
By using the isolated nature of $u_0$, we can show that for all sufficiently small $\e$, $u_{\e, p}$ lies in the interior of $B$. This shows that $u_{\e, p}$ is an $L^1$-local minimizer of $E_{\e, p}$. The same argument shows that $u_{\e, p}$ converges to $u_0$ in $L^1(\Omega)$. By the liminf inequality in Gamma-convergence, we find
$$\liminf E_{\e, p} (u_{\e, p})\geq E_{p}(u_0).$$
Hence,  $\lim_{\e\rightarrow 0} E_{\e, p}(u_{\e, p}) = E_{p}(\Gamma).$
Thus, by Theorem \ref{thm-ACp}, we have
\begin{equation*}
\lim_{\varepsilon\rightarrow 0}\delta^{2} E_{\varepsilon, p}(u_{\varepsilon, p},\eta,\zeta) = c_p\left(
 \delta^{2}E(\Gamma,\eta,\zeta) + (p-1)\int_{\Gamma} (\stackrel{\rightarrow}{n},\stackrel{\rightarrow}{n}\cdot\nabla\eta)^2 d\mathcal{H}^{N-1}\right).\end{equation*} 
The result now follows by combining the above equation with (\ref{SV_eq}) in Proposition \ref{all_var}.
\end{proof}
\begin{proof}[Proof of Corollary \ref{eigen_cor}]

Let denote by $Q_{\e}$ the quadratic function associated to the operator $-\e \Delta + 2\e^{-1}(3u_{\e}^2-1)$, that is, for $\varphi\in C^{1}_{c}(\Omega)$, we have
$$Q_{\e}(u)(\varphi)=\int_{\Omega} \left(\e |\nabla \varphi|^2 + 2\e^{-1} (3u_\e^2-1)\varphi^2\right) dx\equiv d^2 E_{\e}(u_{\e}, \varphi).$$
Similarly, we can define $Q$ for $E$. In particular, for $\varphi\in C^{1}_{c}(\Gamma)$, we have
$$Q(\varphi)= \int_{\Gamma} \left(\abs{\nabla^{\Gamma} \varphi}^2 - \abs{A}^2 \varphi^2\right) d\mathcal{H}^{N-1}.$$
We can naturally extend $Q$ to be defined for compactly supported vector fields in $\Omega$ that are generated by functions defined on $\Gamma$ as follows. Given $f\in C^{1}_{c}(\Gamma)$, let $\eta = f \stackrel{\rightarrow}{n}$ be a normal vector field defined on $\Gamma$. 
Assuming the smoothness of $\Gamma$, we can find an extension $\tilde{\eta}$ of $\eta$ to $\Omega$ such that $(\stackrel{\rightarrow}{n}, 
\stackrel{\rightarrow}{n}\cdot\nabla\tilde{\eta}) =0$. Then, define $Q(\tilde{\eta}):= Q(f).$

For any vector field $V$ defined on $\Gamma$ and is normal to $\Gamma$, we also denote by $V$ its extension to $\Omega$
in such a way that $(\stackrel{\rightarrow}{n}, \stackrel{\rightarrow}{n}\cdot \nabla V) =0.$ As a consequence, (\ref{discrep}) becomes
\begin{equation}\lim_{\e\rightarrow 0} Q_{\e}(\nabla u_{\e}\cdot V) = c_2 Q(V).
\label{polarident}
\end{equation}
By the definition of $\lambda_k$, we can find $k$ linearly independent, orthonormal vector fields $V^{1} = v^{1}\stackrel{\rightarrow}{n},\cdots, V^{k}= v^{k}\stackrel{\rightarrow}{n}$ which are defined on $\Gamma$ and normal to $\Gamma$ such that
\begin{equation}\label{V_ortho} 
\int_{\Gamma} v^i v^j d\mathcal{H}^{N-1}=\delta_{ij}~\text{and}~
\max_{\sum_{i=1}^{k} a^2_{i}=1} Q(\sum_{i=1}^{k} a_{i} V^{i})\leq \lambda_k. 
\end{equation}
Denote $$V^{i}_{\e}  = \left.\frac{d}{dt}\right\rvert_{t=0} u_{\e} \left(\left(x+ tV^{i}(x)\right)^{-1}\right)=-\nabla u_{\e}\cdot V^i.$$ 
As in \cite{Le}, the map $V\longmapsto -\nabla u_{\e}\cdot V$ is linear and one-to-one for $\e$ small.  Thus, the linear independence of $V^{i}$ implies that of $V^{i}_{\e}$ for $\e$ small. Therefore, the $V^{i}_{\e}$ span a space of dimension $k$. It follows from the variational characterization of $\lambda_{\e, k}$ that
\begin{equation}\displaystyle
\sup_{\sum_{i=1}^{k} a^2_{i}=1} \frac{Q_{\e}(\sum_{i=1}^k a_i V^{i}_{\e})}{\e\int_{\Omega} |\sum_{i=1}^k a_i V^{i}_{\e}|^2}\geq \frac{\lambda_{\e, k}}{\e}.
\end{equation}
Take any sequence $\e\rightarrow 0$ such that
$$\frac{\lambda_{\e, k}}{\e}\rightarrow \limsup_{\e\rightarrow 0}\frac{\lambda_{\e, k}}{\e}:= \gamma_k.$$ Then, for any $\delta>0$, we can find $a_1, \cdots, a_k$ with $\sum_{i=1}^k a_i^2=1$ such that for $\e$ small enough
\begin{equation} \frac{Q_{\e}(\sum_{i=1}^k a_i V^{i}_{\e})}{\e\int_{\Omega} |\sum_{i=1}^k a_i V^{i}_{\e}|^2} \geq \gamma_k -\delta.
\label{Qeq1}
\end{equation}
By polarizing (\ref{polarident}) as in \cite{Le}, we have for all $a_{i}$
\begin{equation}
\lim_{\e\rightarrow 0}  Q_{\e} (\sum_{i=1}^{k} a_{i} V_{\e}^{i}) =  c_2 Q (\sum_{i=1}^{k} a_{i} V^{i}) 
\label{Qeq2}
\end{equation}
and the convergence is uniform with respect to $\{a_{i}\}$ such that $\sum_{i=1}^{k} a^2_{i} =1$. Next, we study the convergence of the denominator of the left hand side of (\ref{Qeq1}) when $\e\rightarrow 0$.
By (\ref{2nu0}), we have
\begin{equation}\lim_{\e\rightarrow 0} \e\int_{\Omega} |\sum_{i=1}^k a_i V^{i}_{\e}|^2dx = \lim_{\e\rightarrow 0} \e\int_{\Omega} \sum_{i, j=1}^k a_i a_j (\nabla u^{\e}\cdot V^i)(\nabla u^{\e}\cdot V^i)dx\nonumber = c_2\sum_{i,j=1}^k a_i a_j \int_{\Gamma} v^i v^j d\mathcal{H}^{N-1}=c_2,
\label{Qeq3}
\end{equation}
where we used the first equation in (\ref{V_ortho}) in the last equation.
Combining (\ref{Qeq1})-(\ref{Qeq3}) together with (\ref{V_ortho}), we find that
$$ \gamma_k -\delta\leq Q(\sum_{i=1}^{k} a_{i} V^{i})\leq \lambda_k.$$
Therefore, by the arbitrariness of $\delta$, we have
$\gamma_k\leq \lambda_k,$ proving the Corollary.
\end{proof}

\begin{proof}[Proof of Theorem \ref{isolate3}]
{\it Proof of part (i).} Note that, for $t$ small, $\Phi_t$ defined by (\ref{map-deform}) is a diffeomorphism of $\RR^N$ into itself. We compute 
$$|\tilde E_t|=\int_{\Phi_t(E_0)}dy = \int_{E_0} \abs{\text{det}\nabla\Phi_{t}(x)}dx. $$
We use the following identity for matrices $A$ and $B$
\begin{equation*}
 \text{det}(I + tA + \frac{t^{2}}{2}B) = 1 + t\text{trace}(A) + \frac{t^2}{2}[\text{trace}(B) + (\text{trace}(A))^2 - \text{trace}(A^2)] + O(t^3).
\end{equation*}
Therefore, since for $t$ sufficiently small, $\text{det}\nabla\Phi_{t}(x)>0$,
\begin{multline}
\abs{\text{det}\nabla\Phi_{t}(x)}=\text{det}\nabla\Phi_{t}(x) =\text{det} (I + t\nabla \eta (x) +\frac{t^2}{2}\nabla \zeta)\\= 1 +  t\text{div} \eta + \frac{t^2}{2}[ \text{div}\zeta + (\text{div}\eta)^2 - \text{trace}((\nabla\eta)^2)] + O(t^3).
\label{det_expand}
\end{multline}
It follow that, for small $t$, we have
$$|\tilde E_t| = \int_{E_0} \{1 +  t\text{div} \eta + \frac{t^2}{2}[ \text{div}\zeta + (\text{div}\eta)^2 - \text{trace}((\nabla\eta)^2)] + O(t^3)\} dx .$$
The domain admissibility of $\tilde E_t$ is equivalent to
$$\int_{E_0} \text{div} \eta~ dx =0,~\text{and}~\int_{E_0} [ \text{div}\zeta + (\text{div}\eta)^2 - \text{trace}((\nabla\eta)^2)]dx=0.$$
For any $\eta$, by (\ref{GoodIden}), we can choose
$\zeta = \zeta^{\eta}:= - (div\eta)\eta + (\eta\cdot\nabla)\eta$
so that the second equation holds. Thus, the admissibility of $\tilde E_t$ is reduced to the first equation. This is what we need to prove.  
Hence, one particular second inner variation of the area functional $E(\Gamma)$ with volume constraint (\ref{volc}) and velocity $\eta$ is
$\delta^2 E (\Gamma, \eta, \zeta^{\eta}).$
\\ 
{\it Proof of part (ii).} Let us now consider the special case where $E_0$ is stationary for the area functional $E$ with volume constraint (\ref{volc}), $\eta$ is a smooth vector field  tangent to $\partial\Omega$, normal to $\Gamma$ with $(\stackrel{\rightarrow}{n}, \stackrel{\rightarrow}{n}\cdot\nabla \eta)=0$ on $\Gamma$ and
$\displaystyle\int_{\Gamma} \eta(x)\cdot \stackrel{\rightarrow}{n}(x) d\mathcal{H}^{N-1}(x)=0.$ In this case, by the tangency of $\eta$ to $\partial\Omega$ and the divergence theorem, we have
$$\int_{E_0} div \eta dx=\int_{\Gamma} \eta \cdot \stackrel{\rightarrow}{n} d\mathcal{H}^{N-1}.$$
Thus $\tilde E_t$ is domain admissible and hence $\delta^2 E (\Gamma, \eta, \zeta^{\eta})$ makes sense. 

Applying the stationary condition to the admissible family $\{\Psi_t (E_0)\}$ as in \cite{SZ2} where $\Psi_t$ is defined by (\ref{SZpsi}), we find that the mean curvature $\kappa$ of $\Gamma$ is a constant and that $\partial\Gamma$ is orthogonal to $\partial\Omega$. On $\Gamma$, let $\xi= \eta\cdot \stackrel{\rightarrow}{n}$.
Now, we can compute
$$\int_{\Gamma} \sum_{i=1}^{N-1}\abs{(D_{\tau_{i}}\eta)^{\perp}}^2 d\mathcal{H}^{N-1}=\int_{\Gamma}|\nabla_{\Gamma}\xi|^2d\mathcal{H}^{N-1},$$
$$\int_{\Gamma} \sum_{i,j=1}^{N-1}(\tau_{i}\cdot D_{\tau_{j}}\eta)(\tau_{j}\cdot D_{\tau_{i}}\eta)d\mathcal{H}^{N-1}=\int_{\Gamma}|A_{\Gamma}|^2|\xi|^2 d\mathcal{H}^{N-1}.$$
Since $\eta = \xi\stackrel{\rightarrow}{n}$ on $\Gamma$, we find that
$$div^{\Gamma} \eta= D_{\tau_i}(\xi\stackrel{\rightarrow}{n})\cdot \tau_i= \xi(D_{\tau_i} \stackrel{\rightarrow}{n})\cdot \tau_i= \kappa \xi.$$
Using $(\stackrel{\rightarrow}{n}, \stackrel{\rightarrow}{n}\cdot\nabla \eta)=0$ on $\Gamma$, we find that $div \eta = div^{\Gamma}\eta =\kappa \xi.$ Hence, similarly as above, we obtain
$$div^{\Gamma}((div\eta) \eta)= \kappa (div\eta\eta\cdot  \stackrel{\rightarrow}{n})=\kappa^2|\xi|^2.$$ 
Thus
$$\int_{\Gamma} div^{\Gamma}\left((div\eta)\eta\right)d\mathcal{H}^{N-1}=\int_{\Gamma} \kappa^2|\xi|^2d\mathcal{H}^{N-1}.$$
Note that the vector field $(\eta\cdot\nabla)\eta$ corresponds to the vector field $Z$ in \cite{SZ2}. Computing as in \cite{SZ2} and using the orthogonality of $\partial\Gamma$ and $\partial\Omega$ which is due $E_0$ being stationary, we get
$$\int_{\Gamma} div^{\Gamma} ((\eta\cdot\nabla)\eta)d\mathcal{H}^{N-1}=-\int_{\partial\Gamma\cap\partial\Omega} A_{\partial\Omega}(\stackrel{\rightarrow}{n}, \stackrel{\rightarrow}{n})|\xi|^2 d\mathcal{H}^{N-2}.$$
Hence, with $\zeta^{\eta}= -(div \eta)\eta + (\eta\cdot\nabla)\eta$, we find
$$\int_{\Gamma} div^{\Gamma} \zeta^{\eta}d\mathcal{H}^{N-1}= -\int_{\Gamma} \kappa^2|\xi|^2d\mathcal{H}^{N-1}-\int_{\partial\Gamma\cap\partial\Omega} A_{\partial\Omega}(\stackrel{\rightarrow}{n}, \stackrel{\rightarrow}{n})|\xi|^2 d\mathcal{H}^{N-2}.$$
By (\ref{SVE}) and combining all the above identities, we finally obtain
$$\delta^2 E (\Gamma, \eta, \zeta^{\eta})=\int_{\Gamma}\{div^{\Gamma}\zeta^{\eta} + \kappa^2 |\xi|^2 + |\nabla_{\Gamma}\xi|^2 -|A_{\Gamma}|^2|\xi|^2  \}d\mathcal{H}^{N-1} = J(\xi)=J(\eta\cdot \stackrel{\rightarrow}{n}).$$

Suppose now that $E_0$ is stable for the area functional $E$ with volume constraint (\ref{volc}). Then, by \cite{SZ2}, we know that $J(\xi)\geq 0$. Here, we give another proof using inner variations. For the purpose of calculating the second variation of $E$ as done in \cite{SZ2}, we need an admissible family $E_t$ of deformations of $E_0$ that stay inside $\Omega$. It is natural to consider
$$E_t=\{y\in\Omega: u_t(y)=1\}=\Phi_t(E_0)\cap\Omega.$$
In view of the change of variables, (\ref{det_expand}) and $\zeta= \zeta^{\eta}$, we have
\begin{equation}|\Phi_t(\Omega)| = \int_{\Omega}\{1 + t div \eta + O(t^3)\}dx =|\Omega| + o(t^2).
\label{volt2}
\end{equation}
That the coefficient of $t$ vanishes can be seen from the divergence theorem and the tangency of $\eta$ to $\partial\Omega$. By the domain admissibility of $\tilde E_t$, we have 
$|\tilde E_t| = |E_0| + o(t^2).$
Hence, the admissibility of $E_t$ follows from the following claim.
\begin{claim}
$|\Omega\backslash \Phi_t(\Omega)| + |\Phi_t(\Omega)\backslash \Omega|= o(t^2). $
\label{vol_claim}
\end{claim}
By virtue of the Inverse Function Theorem, we can see that $\Omega\backslash \Phi_t(\Omega)\cup\Phi_t(\Omega)\backslash \Omega$ consists of domains around the boundary $\partial\Omega$. By choosing the extension $\eta$ of $\xi \stackrel{\rightarrow}{n}$ to be $0$ outside a compact set containing $\Gamma$ in $\overline{\Omega}$, we can make sure that the number of the above domains is finite. This extension does not change the quantity $J(\xi)$. Using (\ref{volt2}), it suffices to prove Claim \ref{vol_claim}  for the case when $\Omega\backslash \Phi_t(\Omega)\neq\emptyset$ and $\Phi_t(\Omega)\backslash \Omega\neq \emptyset.$ In this case, we only need to show that
$$|\Omega\backslash \Phi_t(\Omega)| = o(t^2).$$
Suppose $\Omega_i (i\in I)$ are components of $\Omega$ such that $\Phi_t(\Omega_i)\subset \RR^n\backslash \overline{\Omega}.$ We modify the normal component $\xi$ of $\eta$ on $\Gamma$ to be $\tilde\xi$ such that $\tilde\xi =0$ on
$(\cup \Omega_i)\cap E_0$ and $\int_{\Gamma} \tilde\xi d\mathcal{H}^{N-1} =0.$ This can be done by modifying the value of $\xi$ in a compact set $K\subset\subset \Gamma$. We 
extend $\tilde\xi$ to vector field $\tilde \eta$ on $\RR^N$ having properties similar to $\eta$. Let
$$\tilde \Phi_t(x)= x + t\tilde\eta(x) + \frac{t^2}{2}\zeta^{\tilde \eta}(x).$$
Then, $\tilde\Phi_t(\Omega)= \Phi_t(\Omega)\cap\Omega\subset\Omega.$ Moreover, 
$|\tilde\Phi_t(\Omega)| = |\Omega| + o(t^2).$
Therefore the claim follows from
$$|\Omega\backslash \Phi_t(\Omega)| = |\Omega\backslash \tilde\Phi_t(\Omega)|=o(t^2).$$

With Claim \ref{vol_claim}, we can finish the proof of the Poincar\'e inequality. Indeed, since $\Phi_t(\Gamma)\supset \partial E_t\cap\Omega$ with equality when $t=0$, we find that
$$J(\xi)=\delta^2 E(\Gamma, \eta, \zeta^{\eta})=\left.\frac{d^2}{dt^2}\right\rvert_{t=0} \mathcal{H}^{N-1}(\Phi_t(\Gamma))\geq \left.\frac{d^2}{dt^2}\right\rvert_{t=0} \mathcal{H}^{N-1}(\partial E_t\cap\Omega)\geq 0.$$
The first inequality is a relation between our particular second inner variation and the one particular second variation in the sense of Sternberg and Zumbrun \cite{SZ2} while the second inequality follows from the stability for $E_0$. Hence the Poincar\'e inequality follows.\\
{\it Proof of part (iii)}. Let $u_{\e}$ and $u_{0}$ be as in Theorem \ref{isolate2} but now equipped with the volume constraint (\ref{volc}). In the presence of a volume constraint, the first variation of $E_{\e}$ satisfies
$\e \Delta u_{\e}- \e^{-1}W^{'}(u_{\e}) =\lambda_{\e}$
where $\lambda_{\e}$ is the (constant) Lagrange multiplier and for all $\varphi\in C_{0}^{1}(\Omega)$, we have
$$dE_{\e}(u_{\e}, \varphi)=\lambda_{\e}\int_{\Omega}\varphi dx.$$
With the perturbation vector field $\eta^{\e}$, we define
$$\zeta^{\e} = - (div\eta^{\e})\eta^{\e} + (\eta^{\e}\cdot\nabla)\eta^{\e}, ~\Phi_{\e, t} (x) = x + t\eta^{\e}(x) + \frac{t^2}{2}\zeta^{\e}(x).$$ 
We remark that the family $\{ u_{\e}(\Phi_{\e, t}^{-1}(x))\}$ preserves the mass of $u_{\e}$ up to second order in $t$. Indeed, using a change of variables and (\ref{det_expand}), we find that
$$\left.\frac{d}{dt}\right\rvert_{t=0} \int_{\Omega}  u_{\e}(\Phi_{\e, t}^{-1}(x))dx= \int_{\Omega}u_{\e}(x) div \eta^{\e}(x) dx=0$$
and
$$\left.\frac{d^2}{dt^2}\right\rvert_{t=0} \int_{\Omega}  u_{\e}(\Phi_{\e, t}^{-1}(x)) dx=\int_{\Omega}u_{\e}(x) [\text{div}\zeta^{\e} + (\text{div}\eta^{\e})^2 - \text{trace}((\nabla\eta^{\e})^2)] dx=0.$$
Formula (\ref{ut_expand}) in the proof of Proposition \ref{all_var} gives
$$ u_{\e}(\Phi_{\e, t}^{-1}(y)) = u_{\e} (y) -t\nabla u_{\e} \cdot \eta^{\e} + \frac{t^2}{2} X_{\e} + O(t^3)
$$
where using formula (\ref{Xzero}), and taking into account the choice of $\zeta^{\e}$, we have
$$X_{\e}= (D^2 u_{\e}(y) \cdot \eta^{\e}(y), \eta^{\e} (y)) + (\nabla u_{\e}(y), (\eta^{\e}\cdot\nabla ) \eta^{\e} (y) + div (\eta^{\e}) \eta^{\e})= div ((\nabla u_{\e}\cdot \eta^{\e})\eta^{\e}).$$
Using the divergence theorem and the fact that $\eta^\e=0$ on $\partial\Omega$, we get
\begin{equation}\int_{\Omega} X_{\e}dx = \int_{\partial\Omega} (\nabla u_{\e}\cdot\eta^{\e})(\eta^{\e}\cdot\nu)d\mathcal{H}^{N-1}=0,
\label{Xvanish}
\end{equation}
where $\nu$ is the unit outer normal on $\partial\Omega$. 

Using the relation (\ref{SV_eq}) between different notions of variations in Proposition \ref{all_var} for the functional $E_{\e}$ with velocity vector field $\eta^{\e}$ and acceleration vector field $\zeta^{\e}$, we obtain
$$d^2 E_{\e} (u_{\e}, -\nabla u_{\e}\cdot\eta^{\e})=\delta^2 E_{\e} (u_{\e}, \eta^{\e}, \zeta^{\e})-dE_{\e}(u_{\e}, X_{\e})=\delta^2 E_{\e} (u_{\e}, \eta^{\e}, \zeta^{\e}) -\lambda_{\e}\int_{\Omega} X_{\e}dx.$$
Thus, by (\ref{Xvanish}), we obtain
\begin{equation}d^2 E (u_{\e},  -\nabla u_{\e}\cdot\eta^{\e})= \delta^2 E_{\e} (u_{\e}, \eta^{\e}, \zeta^{\e}).
\label{2svm}
\end{equation}
Using
$$\lim_{\e\rightarrow 0}\|\eta^{\e}-\eta\|_{C^{2}(\overline{\Omega})}=0$$
and the explicit formula for $\zeta^{\eta}$ and $\zeta^{\e}$ in terms of $\eta$ and $\eta^{\e}$, we find that
$$\lim_{\e\rightarrow 0}\|\zeta^{\e}-\zeta^{\eta}\|_{C^{1}(\overline{\Omega})}=0.$$
Combining these last two limits with the uniform boundedness of $E_{\e}(u_{\e})$ and the formula for $\delta^2 E_{\e} (u_{\e}, \eta^{\e}, \zeta^{\e})$ in (\ref{svep-p}) with $p=2$, we conclude that
$$\lim_{\e\rightarrow 0}\delta^2 E_{\e} (u_{\e}, \eta^{\e}, \zeta^{\e})= \lim_{\e\rightarrow 0}\delta^2 E_{\e} (u_{\e}, \eta, \zeta^{\eta}).$$
Note that, as in the proof of Theorem \ref{isolate2}, we have $\lim_{\e\rightarrow 0} E_{\e}(u_\e) = E_{2}(\Gamma)=c_2 E(\Gamma).$
As a consequence, we obtain from (\ref{2svm}) and Theorem \ref{thm-ACp} the desired formula 
$$\lim_{\e\rightarrow 0}d^2 E (u_{\e},  -\nabla u_{\e}\cdot\eta^{\e}) = \lim_{\e\rightarrow 0}\delta^2 E_{\e} (u_{\e}, \eta, \zeta^{\eta})=c_2\left\{
 \delta^{2}E(\Gamma,\eta,\zeta^{\eta}) +  \int_{\Gamma} (\stackrel{\rightarrow}{n},\stackrel{\rightarrow}{n}\cdot\nabla\eta)^2 d\mathcal{H}^{N-1}\right\}.$$

\end{proof}

It remains to prove Proposition \ref{all_var}.
\begin{proof}[Proof of Proposition \ref{all_var}]
Our proof goes by explicitly computing all variations and inner variations. We will write
$F= F(z, {\bf p})$
for $z\in \RR$ and ${\bf p} =(p_1, \cdots, p_N)\in \RR^N$. We also set $\nabla_{\bf p} F = (F_{p_1}, \cdots, F_{p_N})$ and $u_t(y) = u(\Phi_t^{-1}(y))$. \\
{\it Usual variations.}
Carrying out the computation of $\left.\frac{d}{dt}\right\rvert_{t=0} A(u + t\varphi)$, and integrating by parts, we find that the first variation of $A$ at $u$ with respect to $\varphi\in C_{c}^{1}(\Omega)$ is given by
\begin{eqnarray}dA(u,\varphi)=\left.\frac{d}{dt}\right\rvert_{t=0} A(u + t\varphi)=\int_{\Omega}\left(F_{z}\varphi + F_{p_i}\varphi_i\right) dx=\int_{\Omega}\left(F_{z} - (F_{p_i})_{x_i}\right)\varphi dx.
\label{fveq}
\end{eqnarray}
The second variation of $A$ at $u$ with respect to $\varphi$ is 
\begin{equation}d^2 A(u,\varphi)=\left.\frac{d^2}{dt^2}\right\rvert_{t=0} A(u + t\varphi)=\int_{\Omega}\left( F_{zz}\varphi^2 + 2 F_{z p_i} \varphi\varphi_i + F_{p_i p_j}\varphi_i \varphi_j\right) dx.
\label{sveq}
\end{equation}
{\it Inner variations.}
The proof is based on the following formula
\begin{equation}u_t(y) = u (y) -t\nabla u \cdot \eta + \frac{t^2}{2} X_0 + O(t^3)
\label{ut_expand}
\end{equation}
where $X_0$ is given by (\ref{Xzero}).
We indicate how to derive this formula. Recalling the definition of $\Phi_t$ in (\ref{map-deform}),
we have
$$x= \Phi_t (\Phi^{-1}_t(x))= \Phi_t^{-1}(x) + t \eta (\Phi^{-1}_t(x)) + \frac{t^2}{2} \zeta (\Phi^{-1}_t(x)).$$
Differentiating both sides with respect to $t$, one gets
$$0=\frac{d}{dt}\Phi^{-1}_t(x) + t\nabla\eta \frac{d}{dt}\Phi^{-1}_t(x) + \eta (\Phi^{-1}_t(x)) + t\zeta (\Phi^{-1}_t(x)) + \frac{t^2}{2} \nabla \zeta \frac{d}{dt}\Phi^{-1}_t(x),$$
and 
\begin{multline*}0=\frac{d^2}{dt^2}\Phi^{-1}_t(x) + \nabla\eta \frac{d}{dt}\Phi^{-1}_t(x) + t \frac{d}{dt}(\nabla\eta \frac{d}{dt}\Phi^{-1}_t(x))+ \nabla\eta \frac{d}{dt}\Phi^{-1}_t(x) \\ + \zeta (\Phi^{-1}_t(x)) + t\frac{d}{dt}\zeta (\Phi^{-1}_t(x)) + \frac{d}{dt}(\frac{t^2}{2} \nabla \zeta \frac{d}{dt}\Phi^{-1}_t(x)).
\end{multline*}
Thus,  evaluating the last two equations at $t=0$, we get 
\begin{equation*}\frac{d}{dt}\Phi^{-1}_t(x)\mid_{t=0} =-\eta(x);~\frac{d^2}{dt^2}\Phi^{-1}_t(x)\mid_{t=0}=  2\nabla\eta \eta(x) - \zeta(x).
\end{equation*}
Now, view $u_t(y)= u(\Phi^{-1}_t(y))$ as a function of $t$. Then

$$u_t(y)\mid_{t=0}= u(y),~~\frac{d}{dt}u_t(y)\mid_{t=0} = \nabla u \frac{d}{dt}\Phi^{-1}_t(y)\mid_{t=0} =-\nabla u(y)\eta (y),$$
and
\begin{eqnarray*}\frac{d^2}{dt^2}u_t(y)\mid_{t=0}&=& \left(D^2 u (\frac{d}{dt}\Phi^{-1}_t(y), \frac{d}{dt}\Phi^{-1}_t(y)) + \nabla u \frac{d^2}{dt^2}\Phi^{-1}_t(y)\right)\mid_{t=0}\\&=& D^2 u(y) (\eta(y), \eta(y)) + (\nabla u(y), 2\nabla\eta \eta(y) - \zeta(y)),
\end{eqnarray*}
and hence (\ref{ut_expand}) follows from the Taylor expansion of $u_t$ in $t$.\\
By change of variables $y=\Phi_{t}(x)$, we have 
\begin{equation}
 A(u_{t}) = \int_{\Omega} F(u(x),\nabla u\cdot\nabla \Phi_{t}^{-1}(\Phi_{t}(x)) \abs{\text{det}\nabla \Phi_{t}(x)}dx.
\label{rewrite_E}
\end{equation}
We need to expand the right-hand side of the above formula up to the second power of $t$. 
Note that
\begin{equation*}
 \nabla\Phi_{t}^{-1}(\Phi_{t}(x)) = [I + t\nabla\eta (x) +\frac{t^2}{2}\nabla\zeta(x)]^{-1} = I - t\nabla\eta -\frac{t^2}{2}\nabla\zeta(x)+ t^{2}(\nabla\eta)^2 + O(t^3),
\end{equation*}
hence
\begin{equation*}
 \nabla u\cdot\nabla \Phi_{t}^{-1}(\Phi_{t}(x)) = \nabla u - t\nabla u\cdot\nabla\eta -\frac{t^2}{2}\nabla u\cdot\nabla\zeta(x)+ t^{2}\nabla u\cdot(\nabla\eta)^2 + O(t^3).
\end{equation*}
Plugging this equation together with (\ref{det_expand}) into (\ref{rewrite_E}), we find that
$$\delta A(u,\eta,\zeta)= \mathcal{A}_{0}^{'}(0),~\delta^2 A(u,\eta,\zeta)= \mathcal{A}_{0}^{''}(0)$$
where
$$\mathcal{A}_{0}(t) = \int_{\Omega} F(u, \nabla u - t\nabla u\cdot\nabla\eta - t^2 Y) (1 + t div \eta + \frac{t^2}{2}X) dx,$$
with
$$X= div \zeta + (div \eta)^2 -trace (\nabla\eta)^2;~~Y= \frac{1}{2} \nabla u\cdot\nabla\zeta-\nabla u\cdot (\nabla \eta)^2.$$
Let $\eta= (\eta^1, \cdots,\eta^N).$ Then, integrating by parts, we find that the first inner variation is
\begin{eqnarray*}\delta A(u,\eta,\zeta)=\mathcal{A}_{0}^{'}(0) &=& \int_{\Omega}\left( F div \eta - F_{p_i} 
\frac{\partial}{x_i}\eta^j u_j\right) dx \\ &=& \int_{\Omega} [\frac{\partial}{x_j} F  - \frac{\partial}{x_i}( F_{p_i}  u_j] (-\eta^j) dx= \int_{\Omega}[F_z -\frac{\partial }{\partial x_i} F_{p_i}] (- u_j \eta^j)dx\\
&=& dA(u,-\nabla u\cdot\eta),
\end{eqnarray*}
proving (\ref{FV_eq}).

Though not directly used in the proof of our proposition, we include a formula for the second inner variation here because of its many uses in Gamma-converging energies (see \cite{Le} and the proof of Theorem \ref{mainSV}).
The second inner variation is
\begin{equation}\delta^{2} A(u,\eta,\zeta)=\int_{\Omega}\left\{ FX - 2 (\nabla_{\bf p} F, \nabla u\cdot \nabla\eta) div \eta - 2 (\nabla_{\bf p} F, Y) + F_{p_i p_j}(\nabla u\cdot\nabla \eta)^{i}(\nabla u\cdot\nabla \eta)^{j}\right\} dx.
\label{SVformula}
\end{equation}
Indeed, we write
$$\mathcal{A}^{''}_0(0) = \mathcal{A}^{''}_{D}(0)+ \mathcal{A}^{''}_{B}(0) + \int_{\Omega} FX$$
where
$$\mathcal{A}_{D}(t)=\int_{\Omega} F(u, \nabla u - t\nabla u\cdot\nabla\eta - t^2 Y),~\mathcal{A}_{B}(t)= \int_{\Omega} F(u, \nabla u - t\nabla u\cdot\nabla\eta - t^2 Y) t div\eta.$$
We note that
$$\mathcal{A}^{'}_{B}(t)= \int_{\Omega} F(u, \nabla u - t\nabla u\cdot\nabla\eta - t^2 Y) div \eta + \int_{\Omega} \frac{d}{dt}F(u, \nabla u - t\nabla u\cdot\nabla\eta - t^2 Y) t div \eta.$$
Therefore
\begin{eqnarray*}\mathcal{A}^{''}_{B}(0)=\int_{\Omega} 2 \frac{d}{dt}F(u, \nabla u - t\nabla u\cdot\nabla\eta - t^2 Y)div\eta\mid_{t=0} = \int_{\Omega}-2 (\nabla_{\bf p} F,\nabla u\cdot \nabla\eta)div\eta.
\end{eqnarray*}
Now, we have
$$\mathcal{A}^{'}_{D}(t) = \int_{\Omega}-F_{p_i}(u, \nabla u - t\nabla u\cdot\nabla\eta - t^2 Y) (\frac{\partial}{x_i}\eta^j u_j + 2t Y^i).$$
Therefore
\begin{eqnarray*}\mathcal{A}^{''}_{D}(0)&=& \int_{\Omega} F_{p_i p_k} (\frac{\partial}{x_i}\eta^j u_j)(\frac{\partial}{x_k}\eta^l u_l) - 2F_{p_i} Y^i\\ &=& \int_{\Omega} -2(\nabla_{\bf p} F, Y) + F_{p_i p_j} ((\nabla u\cdot\nabla\eta)^i, (\nabla u\cdot\nabla \eta)^j).
\end{eqnarray*}
We observe, using (\ref{ut_expand}), that the second inner variation is also equal to the second derivative of the following function at $0$
$$\mathcal{A}_{1}(t) = \int_{\Omega} F(u-t \nabla u\cdot\eta + \frac{t^2}{2}X_0, \nabla u - t\nabla (\nabla u\cdot\eta) +\frac{t^2}{2}\nabla X_0) dy.$$
We compute
$$\mathcal{A}_{1}^{'}(t) =\int_{\Omega} F_z (-\nabla u\cdot\eta + t X_0)-F_{p_i} (\frac{\partial}{x_i}(\nabla u\cdot \eta)- t \frac{\partial}{x_i} X_0)$$
and
\begin{multline*}\mathcal{A}_{1}^{''}(t) = \int_{\Omega} F_{zz} (-\nabla u\cdot\eta + t X_0)^2-2F_{zp_i} (\frac{\partial}{x_i}(\nabla u\cdot \eta)- t \frac{\partial}{x_i} X_0)(-\nabla u\cdot\eta + t X_0) \\+ \int_{\Omega}F_{p_i p_j} (\frac{\partial}{x_i}(\nabla u\cdot \eta)- t \frac{\partial}{x_i} X_0)(\partial_{j}(\nabla u\cdot \eta)- t \frac{\partial}{x_j} X_0)+ \int_{\Omega} F_{z} X_0 + F_{p_i} \frac{\partial}{x_i} X_0.
\end{multline*}
It follows that
\begin{multline*}\mathcal{A}_{1}^{''}(0) =\int_{\Omega} F_{zz} (\nabla u\cdot \eta)^2 + 2 F_{zp_i} (\nabla u\cdot\eta)\frac{\partial}{x_i}(\nabla u\cdot \eta) \\+ \int_{\Omega}F_{p_i p_j} \frac{\partial}{x_i} (\nabla u \cdot \eta)\frac{\partial}{x_j}(\nabla u \cdot \eta) + \int_{\Omega} F_{z} X_0 + F_{p_i} \frac{\partial}{x_i}X_0.
\end{multline*}
Comparing the above formula with (\ref{sveq}) and (\ref{fveq}), we find that
$$\delta^2 A(u,\eta, \zeta)=\mathcal{A}_{1}^{''}(0) = \mathcal{A}^{''}(0) (\nabla u\cdot\eta) + \mathcal{A}^{'}(0) (X_0) = d^2 A(u, -\nabla u\cdot \eta) + dA(u, X_0),$$
proving (\ref{SV_eq}).
\end{proof}

\section{p-Laplace Allen-Cahn functionals }
\label{sec-ACp}
In this section, we prove Theorem \ref{thm-ACp}. For $1<p<\infty$, let $q=\frac{p}{p-1}$ be its conjugate. Then
$$E_{\e, p} (u_\e)=\int_{\Omega} \left( \frac{\e^{p-1} |\nabla u_\e|^p}{p} + \frac{W(u_\e)}{q\e}\right) dx.$$
The hypotheses of Theorem \ref{thm-ACp} gives that
\begin{equation}\lim_{\e\rightarrow 0} E_{\e, p}(u_\e) = c_p \mathcal{H}^{n-1}(\Gamma):= E_{p}(\Gamma)
\label{wellprep}
\end{equation}
and that $u_\e\rightarrow u_0$ in $L^{1}(\Omega)$ with $\Gamma$ being the interface between the phases $\pm 1$ of $u_0.$
Let
$$a_\e(x):= \e^{\frac{p-1}{p}}|\nabla u_\e|,~ b_\e(x):= \frac{W^{\frac{p-1}{p}}(u_\e)}{\e^{\frac{p-1}{p}}},~\Phi (t) = \int_{0}^{t} W^{\frac{p-1}{p}}(s) ds.$$
Then, we have the following simple but very useful relations. 
\begin{lemma} We have
\begin{equation}
\label{ener1}
\lim_{\e\rightarrow 0}\int_{\Omega} [\frac{a_{\e}^p}{p}+ \frac{b_{\e}^q}{q}- a_\e b_\e] dx =0
\end{equation}
and
\begin{equation}\lim_{\e\rightarrow 0}\int_{\Omega}|\nabla \Phi (u_\e)| dx= \int_{\Omega}|\nabla \Phi (u_0)|dx.
\label{ener1bis}
\end{equation}
\label{lem-ener}
\end{lemma}
\begin{proof}
By Young's inequality, 
$$\frac{\e^{p-1} |\nabla u_\e|^p}{p} + \frac{W(u_\e)}{q\e} = \frac{[\e^{\frac{p-1}{p}} |\nabla u_\e|]^p}{p} + \frac{1}{q}[\frac{W^{\frac{p-1}{p}}(u_\e)}{\e^{\frac{p-1}{p}}}]^q
 \geq |\nabla u_\e| W^{\frac{p-1}{p}}(u_\e)= |\nabla \Phi (u_\e)|.$$
By lower semicontinuity and the coarea formula,
\begin{eqnarray*} \liminf_{\e\rightarrow 0} E_{\e, p}(u_\e) = \liminf_{\e\rightarrow 0} \int_{\Omega}[\frac{a_{\e}^p}{p}+ \frac{b_{\e}^q}{q}]dx &\geq& \liminf_{\e\rightarrow 0}  \int_{\Omega} |\nabla \Phi (u_\e)|dx\\ &\geq& \int_{\Omega} |\nabla \Phi (u_0)|dx= \mathcal{H}^{n-1}(\Gamma)\times (\Phi (1) -\Phi (-1)) \\ &=& \mathcal{H}^{n-1}(\Gamma) \int_{-1}^{1} (W(s))^{\frac{p-1}{p}} ds \equiv c_p \mathcal{H}^{n-1}(\Gamma).
\end{eqnarray*}
Combining the above inequalities with (\ref{wellprep}), we conclude that (\ref{ener1}) and (\ref{ener1bis}) hold. 
\end{proof}
We now show the following equi-partition of energy.
\begin{lemma}We have 
\begin{equation}\lim_{\e\rightarrow 0} \int_{\Omega} |\e^{p-1} |\nabla u_\e|^p - \frac{W(u_\e)}{\e}| dx =\lim_{\e\rightarrow 0} \int_{\Omega}|a_{\e}^p- b_{\e}^{q}|dx=0
\label{equi-ab}
\end{equation}
and
\begin{equation}\lim_{\e\rightarrow 0}\int_{\Omega}|\e^{p-1} |\nabla u_\e|^p- |\nabla \Phi (u_\e)|| dx= \lim_{\e\rightarrow 0}\int_{\Omega}|a_{\e}^{p}- a_{\e}b_{\e}|dx=0.
\label{equi-abphi}
\end{equation}
\label{propo-equi}
\end{lemma}

\begin{proof}[Proof of Lemma \ref{propo-equi}]
We need to show that
\begin{equation}\lim_{\e\rightarrow 0}\int_{\Omega}|a_{\e}^p - b_{\e}^q| dx =0.
\label{equi1}
\end{equation}
The roles of $p$ and $q$ can be interchanged in (\ref{ener1}) and (\ref{equi1}) so we can assume for the sake of the proof of (\ref{equi1}) that $p\geq 2.$ 
Let
$a_\e (x) = t_\e(x) b^{\frac{q}{p}}_\e(x).$ Then (\ref{ener1}) gives
\begin{equation}
\label{ener3}
\lim_{\e\rightarrow 0} \int_{\Omega} [\frac{t_{\e}^p}{p} +\frac{1}{q} - t_\e] b_{\e}^q =0.
\end{equation}
Now, by H\"older's inequality, we have
\begin{equation}
\int_{\Omega}|a_{\e}^p - b_{\e}^q|  = \int_{\Omega} b_{\e}^q|t_{\e}^p-1| \leq \left (\int_{\Omega}b_{\e}^q |\frac{t_{\e}^p-1}{t_\e-1}|^q\right)^{\frac{1}{q}}\left(\int_{\Omega}b_{\e}^q |t_{\e}-1|^p\right)^{\frac{1}{p}}.
\label{ener4}
\end{equation}
By using the elementary inequality
$$|\frac{t^p-1}{p}- (t-1)|\geq \frac{1}{p}|t-1|^p~\text{for all}~p\geq 2~\text{and}~t>0, $$
and (\ref{ener3})
\begin{equation}
\label{ener5}
\int_{\Omega}b_{\e}^q |t_{\e}-1|^p  \leq \int_{\Omega} b_{\e}^q |\frac{t_{\e}^p-1}{p}- (t_{\e}-1)|p = 
\int_{\Omega}pb_{\e}^q |\frac{t_{\e}^p}{p} +\frac{1}{q}- t_\e| \rightarrow 0 ~\text{as}~ \e\rightarrow 0.
\end{equation}
By using the elementary inequality
$$|\frac{t^p-1}{t-1}| \leq p (t^{p-1} + 1)~\text{for all}~p\geq 1~\text{and}~t>0, $$ and recalling $q=\frac{p}{p-1}$, we have
\begin{equation}
\label{ener6}
\int_{\Omega} b_{\e}^q \frac{t_{\e}^p-1}{t_\e-1}|^q \leq \int_{\Omega} b_{\e}^q p ^{\frac{p}{p-1}}[t_{\e}^{p-1}+1]^{\frac{p}{p-1}}\leq \int_{\Omega} C_p b_{\e}^{q} (t_{\e}^{p}+ 1) = C_p \int_{\Omega} (b_{\e}^ q + a_{\e}^{p}) \leq C.
\end{equation}
From (\ref{ener4})-(\ref{ener6}), we obtain the equi-partition of energy.\\
To prove (\ref{equi-abphi}), we note that by (\ref{equi-ab}),
$$\int_{\Omega}|a_{\e}^p-a_\e b_\e|= \int_{\Omega}|\frac{a_{\e}^p - b_{\e}^q}{q} + \frac{a_{\e}^p}{p}+ \frac{b_{\e}^q}{q}- a_\e b_\e |\rightarrow 0~\text{as}~\e\rightarrow 0.$$
\end{proof}

\begin{lemma}
We have
\begin{equation}\e^{p-1} \nabla u_{\e}\otimes \nabla u_{\e} |\nabla u_{\e}|^{p-2}\rightharpoonup c_p \stackrel{\rightarrow}{n}\otimes \stackrel{\rightarrow}{n}\mathcal{H}^{N-1}\lfloor \Gamma
\label{2nu}
\end{equation}
and 
\begin{equation}\e^{p-1} \nabla u_{\e}\otimes \nabla u_{\e}\otimes\nabla u_{\e}\otimes |\nabla u_{\e}|^{p-4}\rightharpoonup c_p \stackrel{\rightarrow}{n}\otimes \stackrel{\rightarrow}{n}\otimes \stackrel{\rightarrow}{n} \otimes \stackrel{\rightarrow}{n}\mathcal{H}^{N-1}\lfloor \Gamma.
\label{4nu}
\end{equation}
\label{propo-Res}
\end{lemma}
\begin{proof}
We have
$$\e^{p-1} \nabla u_{\e}\otimes \nabla u_{\e} |\nabla u_{\e}|^{p-2} =\frac{\nabla u_{\e}}{|\nabla u_{\e}|}\otimes \frac{\nabla u_{\e}}{|\nabla u_{\e}|} a_{\e}^p= \frac{\nabla \Phi(u_{\e})}{|\nabla \Phi (u_{\e})|}\otimes \frac{\nabla \Phi(u_{\e})}{|\nabla \Phi(u_{\e})|} a_{\e}^p.$$
Since
$a_{\e}b_{\e}= |\nabla \Phi(u_{\e})|,$
using Lemma \ref{propo-equi}, we find that for any $\varphi\in C^{\infty}_{0}(\Omega),$
\begin{eqnarray*}
\lim_{\e\rightarrow 0}\int_{\Omega}\frac{\nabla \Phi(u_{\e})}{|\nabla \Phi (u_{\e})|}\otimes \frac{\nabla \Phi(u_{\e})}{|\nabla \Phi(u_{\e})|} a_{\e}^p \varphi dx&=& \lim_{\e\rightarrow 0} \int_{\Omega}\frac{\nabla \Phi(u_{\e})}{|\nabla \Phi (u_{\e})|}\otimes \frac{\nabla \Phi(u_{\e})}{|\nabla \Phi(u_{\e})|} |\nabla \Phi (u_{\e})| \varphi dx\\ &=& \int_{\Gamma} c_p \stackrel{\rightarrow}{n}\otimes \stackrel{\rightarrow}{n}\varphi d\mathcal{H}^{N-1}
\end{eqnarray*}
where the last equality follows from Reshetnyak's theorem \cite{Res} (see also the appendix in Luckhaus-Modica \cite{LM}) and the convergence (\ref{ener1bis}) in Lemma 
\ref{lem-ener}
This proves (\ref{2nu}). The proof of (\ref{4nu}) is similar.
\end{proof}
Now, we turn to the second variation formula for $E_{\e, p}$ and complete the proof of Theorem \ref{thm-ACp}.
\begin{proof}[Proof of Theorem \ref{thm-ACp}]
We use (\ref{SVformula}) in the proof of Proposition \ref{all_var} (see also \cite{Le}) to conclude that
\begin{multline}
\delta^{2} E_{\varepsilon}(u_{\varepsilon},\eta,\zeta) =\int_{\Omega} \left\{ \left( \frac{\e^{p-1} |\nabla u_\e|^p}{p} + \frac{W(u_\e)}{q\e}\right) \left(\text{div}\zeta + (\text{div}\eta)^2 -
\text{trace}((\nabla \eta)^2)\right) \right\}\\ 
-2\int_{\Omega}\e^{p-1}|\nabla u_\e|^{p-2}(\nabla u^{\varepsilon},\nabla u^{\varepsilon}\cdot\nabla\eta)\text{div}\eta\\ 
-2\int_{\Omega}\left(\e^{p-1}|\nabla u_\e|^{p-2}\nabla u^{\varepsilon}, \frac{1}{2}\nabla u^{\varepsilon}\cdot\nabla\zeta -\nabla u^{\varepsilon}\cdot (\nabla\eta)^2\right) \\
+ \int_{\Omega}\left(\e^{p-1}|\nabla u_\e|^{p-2}\abs{\nabla u^{\varepsilon}\cdot\nabla\eta}^2 + (p-2) \e^{p-1} (\nabla u_\e)^i (\nabla u_\e)^j|\nabla u_\e|^{p-4} (\nabla u^{\varepsilon}\cdot\nabla\eta)^i (\nabla u^{\varepsilon}\cdot\nabla\eta)^j\right).
\label{svep-p}\end{multline}
By letting $\e\rightarrow 0$ and using Lemma \ref{propo-Res}, we obtain
\begin{multline}\lim_{\varepsilon\rightarrow 0}\delta^{2} E_{\varepsilon, p}(u_{\varepsilon}, \eta, \zeta) =
c_p\int_{\Gamma}\left\{\text{div}\zeta + (\text{div}\eta)^2 -
\text{trace}((\nabla \eta)^2- 2(\stackrel{\rightarrow}{n},\stackrel{\rightarrow}{n}\cdot\nabla\eta) \text{div} \eta \right\}d\mathcal{H}^{N-1}\\ - 
2c_p \int_{\Gamma}(\stackrel{\rightarrow}{n}, \frac{1}{2}\stackrel{\rightarrow}{n}\cdot\nabla\zeta -\stackrel{\rightarrow}{n} \cdot (\nabla\eta)^2 )d\mathcal{H}^{N-1} + 
c_p\int_{\Gamma}\left[|\stackrel{\rightarrow}{n}\cdot \nabla \eta|^2 + (p-2)(\stackrel{\rightarrow}{n},\stackrel{\rightarrow}{n}\cdot\nabla\eta)^2\right]d\mathcal{H}^{N-1}.
\label{SVpp}
\end{multline} 
As in the proof of Theorem 1.1 in \cite{Le} (see (2.8) there), we find that
\begin{multline*}
 \text{div}\zeta + (\text{div}\eta)^2 -
\text{trace}((\nabla \eta)^2- 2(\stackrel{\rightarrow}{n},\stackrel{\rightarrow}{n}\cdot\nabla\eta) \text{div} \eta -
2 (\stackrel{\rightarrow}{n}, \frac{1}{2}\stackrel{\rightarrow}{n}\cdot\nabla\zeta -\stackrel{\rightarrow}{n} \cdot (\nabla\eta)^2 ) + 
|\stackrel{\rightarrow}{n}\cdot \nabla \eta|^2 \\=
\text{div}^{\Gamma}\zeta + (\text{div}^{\Gamma}\eta)^2 + \sum_{i=1}^{N-1}\abs{(D_{\tau_{i}}\eta)^{\perp}}^2 -
 \sum_{i,j=1}^{N-1}(\tau_{i}\cdot D_{\tau_{j}}\eta)(\tau_{j}\cdot D_{\tau_{i}}\eta) +  (\stackrel{\rightarrow}{n},\stackrel{\rightarrow}{n}\cdot\nabla\eta)^2.
\end{multline*}
Using $E_{\e}(\Gamma)= c_p E(\Gamma)$ and the second inner variation for $E$ given by (\ref{SVE}), we find that the right hand side of (\ref{SVpp}) is equal to 
$$\delta^{2}E_p(\Gamma,\eta,\zeta) + c_p\left[\int_{\Gamma} (\stackrel{\rightarrow}{n},\stackrel{\rightarrow}{n}\cdot\nabla\eta)^2d\mathcal{H}^{N-1} + (p-2)\int_{\Gamma}(\stackrel{\rightarrow}{n},\stackrel{\rightarrow}{n}\cdot\nabla\eta)^2 d\mathcal{H}^{N-1}\right].$$
Therefore, we obtain the desired formula stated in the theorem.
\end{proof}

\section{Second inner variations and stability of Ginzburg-Landau}
\label{sec-GLSV}
In this section, we prove Theorem \ref{mainSV}.

\begin{proof}[Proof of Theorem \ref{mainSV}]
Following the method of \cite{Serfaty} and arguing as in \cite{Le} using (\ref{discrep0}) and (\ref{discrep2}), we get the second conclusion of our theorem. Therefore, it remains to prove (\ref{discrep0}).

First of all, we have the following formula for the second inner variation of $E$ at $\Gamma$ (see Simon \cite[p. 51]{Simon}, for example)
\begin{equation}
 \delta^{2}E(\Gamma,\eta,\zeta) = \int_{\Gamma}\left ( \text{div}^{\Gamma}\zeta + (\text{div}^{\Gamma}\eta)^2 + 
\sum_{i=1}^{N-2}\abs{(D_{\tau_{i}}\eta)^{\perp}}^2 - \sum_{i,j=1}^{N-2}(\tau_{i}\cdot D_{\tau_{j}}\eta)(\tau_{j}\cdot D_{\tau_{i}}\eta)\right) d\mathcal{H}^{N-2},
\label{sve}\end{equation}
where $\text{div}^{\Gamma}\varphi$ denotes 
the tangential divergence of $\varphi$ on $\Gamma$;  for each $\tau\in T_{x}(\Gamma)$, $D_{\tau} \eta$ is the directional derivative and the normal part of $D_{\tau_{i}} \eta$ is denoted by 
$
 (D_{\tau_{i}} \eta)^{\perp} = D_{\tau_{i}} \eta -\sum_{j=1}^{N-2}(\tau_{j}\cdot D_{\tau_{i}} \eta)\tau_{j}. $

Let $(\cdot,\cdot)$ denote the inner product on $C^{N}$ identified with $(\RR^{2})^{N}$. This means that, for $a, b\in C^{N}$, we have
$(a, b) =\frac{1}{2} (a\overline{b} + \overline{a}b).$
For the second inner variation of $E_{\varepsilon}$ at $u_{\varepsilon}$, we use (\ref{SVformula}) in the proof of Proposition \ref{all_var} (see also \cite{Le}) to conclude that
\begin{multline}
\delta^{2} E_{\varepsilon}(u_{\varepsilon},\eta,\zeta) =\frac{1}{|log\e|}\int_{\Omega} \left\{ \left( \frac{\abs{\nabla u_{\varepsilon}}^2}{2} +
\frac{W(u_{\varepsilon})}{4\e^2}\right) \left(\text{div}\zeta + (\text{div}\eta)^2 -
\text{trace}((\nabla \eta)^2)\right) \right.\\ +  \abs{\nabla u_{\varepsilon}\cdot\nabla\eta}^2 + 
2(\nabla u_{\varepsilon},\nabla u^{\varepsilon}\cdot (\nabla\eta)^2) -
(\nabla u_{\e},\nabla u_{\e}\cdot\nabla\zeta)\\ \left.-2 (\nabla u_{\varepsilon},\nabla u_{\varepsilon}\cdot\nabla\eta)\text{div}\eta\right\} dx.
\label{svep}\end{multline}
We will pass the above expression to the limit $\e\rightarrow 0$. To do this, we need to study the convergence properties of $  \frac{1}{\abs{\log \e}}  \nabla u_{\varepsilon}\otimes\nabla u_{\varepsilon}$. 
Let  
$$\alpha_{ij,\e} =\frac{1}{|log\e|}\left(e_{\e}(u_{\e})\delta_{ij}-(\frac{\partial u_{\e}}{\partial x_i},\frac{\partial u_{\e}}{\partial x_j})\right),~\beta_{ij,\e}=\frac{1}{|log\e|}(\frac{\partial u_{\e}}{\partial x_i},\frac{\partial u_{\e}}{\partial x_j}).$$
Then
$$\alpha_{ij,\e} \rightharpoonup \alpha_{ij,\ast},~\beta_{ij,\e} \rightharpoonup \beta_{ij,\ast}~~\text{in the sense of Radon measures}.$$
Since $|\alpha_{ij,\ast}|\leq N \mu_{\ast}$ and $|\beta_{ij,\ast}|\leq N \mu_{\ast}$, we can write
$$\alpha_{ij,\ast}= A_{ij}(x) \mu_{\ast}~\text{and}~\beta_{ij,\ast}= B_{ij}(x) \mu_{\ast}.$$
Then for $\mathcal{H}^{N-2}$-a.e. $x\in \Gamma$, $A(x)= (A_{ij}(x))$ represents the orthogonal projection onto the $(N-2)$-dimensional tangent space $T_{x}\Gamma$ of $\Gamma$ (see \cite[pp. 498--499]{BBO}).  It follows that
\begin{equation}
trace (A)= N-2,~A^2=A.
\label{traceA}
\end{equation}
From the definition of $\alpha_{ij,\e}$, we find that
$$\beta_{ij,\ast} = \mu_{\ast}\delta_{ij}-\alpha_{ij,\ast}= (\delta_{ij}-A_{ij})\mu_{\ast}.$$
Note that the matrix $B(x)= (\delta_{ij}-A_{ij}(x))$ is symmetric and nonnegative definite with
$$trace (B)=2, B^2= B$$
by (\ref{traceA}).  Thus, for  $\mathcal{H}^{N-2}$-a.e. $x\in \Gamma$, we can find
two orthogonal
unit vectors  $\stackrel{\rightarrow}{p}(x)$ and $\stackrel{\rightarrow}{q}(x)$ in the normal space $(T_x\Gamma)^{\perp}$
such that
\begin{equation*}
  \frac{1}{\abs{\log \e}}  \nabla u_{\varepsilon}\otimes\nabla u_{\varepsilon} dx\rightharpoonup \left(
\stackrel{\rightarrow}{p}\otimes\stackrel{\rightarrow}{p} + \stackrel{\rightarrow}{q}\otimes\stackrel{\rightarrow}{q}\right)
\mu_{\ast} .
\end{equation*}
From the connectedness of $\Gamma$, we have by the Constancy Theorem \cite[Theorem 41.1]{Simon}, 
$$\mu_{\ast} = m\pi \mathcal{H}^{N-2}\lfloor\Gamma $$
where $m$ is a positive constant.  
In particular, we have
\begin{equation}
 \lim_{\e\rightarrow 0} E_{\e}(u_{\e}) = m\pi \mathcal{H}^{N-2} (\Gamma)\equiv m\pi E(\Gamma).
\label{appro2}
\end{equation}
The introduction of  the constant $\pi$ in the above equation was inspired by the fact that 
$E_{\e}$ Gamma-converges to 
$E_0(\Gamma):= \pi \mathcal{H}^{N-2} (\Gamma) $ (see \cite{ABO}).
Using (\ref{appro2}), we find
\begin{equation}
  \frac{1}{\abs{\log \e}}  \nabla u_{\varepsilon}\otimes\nabla u_{\varepsilon} dx\rightharpoonup m\pi\left(
\stackrel{\rightarrow}{p}\otimes\stackrel{\rightarrow}{p} + \stackrel{\rightarrow}{q}\otimes\stackrel{\rightarrow}{q}\right)
\mathcal{H}^{N-2}\lfloor\Gamma.
\label{resGL}
\end{equation}

This result is similar to identity (2.10) in \cite{Le}.  
We complexify the orthogonal planes to $\Gamma$ at each point by setting
$
 \stackrel{\rightarrow}{n}^{C} =  \stackrel{\rightarrow}{p} + i  \stackrel{\rightarrow}{q}.
$
Then we have
\begin{equation}
\frac{1}{\abs{\log \e}}  \nabla u_{\varepsilon}\otimes\nabla u_{\varepsilon} dx\rightharpoonup m\pi
\stackrel{\rightarrow}{n}^{C}\otimes\stackrel{\rightarrow}{n}^{C}
\mathcal{H}^{N-2}\lfloor\Gamma.
\label{resGL2}
\end{equation}
Passing to the limit in (\ref{svep}), employing (\ref{appro2}) and (\ref{resGL2}), we obtain
\begin{multline*}
 \lim_{\varepsilon\rightarrow 0}\delta^{2}E_{\varepsilon}(u_{\varepsilon},\eta,\zeta) =m \pi\int_{\Gamma}\left\{\text{div}\zeta + 
(\text{div}\eta)^2 - \text{trace} ((\nabla \eta)^2) \right\} d\mathcal{H}^{N-2}\\+ m\pi\int_{\Gamma}\left\{\abs{\stackrel{\rightarrow}{n}^{C}\cdot\nabla\eta}^2 + 
2(\stackrel{\rightarrow}{n}^{C},\stackrel{\rightarrow}{n}^{C}\cdot(\nabla\eta)^2) -(\stackrel{\rightarrow}{n}^{C}, \stackrel{\rightarrow}{n}^{C}\cdot\nabla\zeta) - 
2 (\stackrel{\rightarrow}{n}^{C},\stackrel{\rightarrow}{n}^{C}\cdot\nabla\eta)\text{div}\eta\right\}d\mathcal{H}^{N-2}.
\end{multline*}
In view of the identity
$
\text{div}^{\Gamma}\eta = \text{div}\eta - (\stackrel{\rightarrow}{n}^{C},\stackrel{\rightarrow}{n}^{C}\cdot\nabla\eta), $ the above equation becomes
                                                                                                          \begin{multline}
\lim_{\varepsilon\rightarrow 0}\delta^{2}E_{\varepsilon}(u_{\varepsilon},\eta,\zeta) = 
m\pi\int_{\Gamma}\left\{ \text{div}^{\Gamma}\zeta + (\text{div}^{\Gamma}\eta)^2 - \text{trace} ((\nabla \eta)^2)\right\}d\mathcal{H}^{N-2} \\+ 
m\pi\int_{\Gamma}\left\{\abs{\stackrel{\rightarrow}{n}^{C}\cdot\nabla\eta}^2 + 2(\stackrel{\rightarrow}{n}^{C},\stackrel{\rightarrow}{n}^{C}\cdot(\nabla\eta)^2) - 
(\stackrel{\rightarrow}{n}^{C},\stackrel{\rightarrow}{n}^{C}\cdot\nabla\eta)^2\right\}d\mathcal{H}^{N-2}.
\label{svreduced}
\end{multline}                                                                     
Some calculation using local coordinates shows that the right hand side of the above equation is the right hand side of (\ref{discrep0}), completing its proof. For 
the reader's convenience, we include the details. By the introduction of  the two vectors $\stackrel{\rightarrow}{p}(x)$ and $\stackrel{\rightarrow}{q}(x)$ in (\ref{resGL}), we can 
choose local coordinates so that $\{\tau_{1},\cdot,\tau_{N-2}, \stackrel{\rightarrow}{p}, \stackrel{\rightarrow}{q}\}$ is the orthonormal basis of $\RR^{N}$;
furthermore, $\stackrel{\rightarrow}{p} = (0,\cdots, 0,1, 0)$ and $\stackrel{\rightarrow}{q} = (0, \cdots, 0, 0,1)$. 
Note that

\begin{equation*}
 (n_{j}^{C}, n_{k}^{C}) = (p_{j} + i q_{j}, p_{k} + i q_{k}) = p_{j}p_{k} + q_{j}q_{k} = \delta_{(N-1)j}\delta_{(N-1)k} + \delta_{Nj}\delta_{Nk}.
\end{equation*}
We calculate successively, omitting
the superscript $C$.\\
(i) $(\nabla\eta)_{ij} = \frac{\partial\eta^{i}}{\partial x_{j}}$,\\
(ii) $((\nabla\eta)^2)_{ij} = \sum_{k}\frac{\partial\eta^{i}}{\partial x_{k}}\frac{\partial\eta^{k}}{\partial x_{j}}$,\\
(iii) $\text{trace} (\nabla\eta)^2 = \sum_{i} ((\nabla\eta)^2)_{ii} = 
\sum_{i,k}\frac{\partial\eta^{i}}{\partial x_{k}}\frac{\partial\eta^{k}}{\partial x_{i}}$,\\
(iv) \begin{eqnarray*}
2 (\stackrel{\rightarrow}{n}, \stackrel{\rightarrow}{n}\cdot (\nabla\eta)^2) = 
2\sum_{i,j}(n_{i},n_{j})((\nabla\eta)^2)_{ij} &=& 
2 ((\nabla\eta)^2)_{(N-1)(N-1)} + 2 ((\nabla\eta)^2)_{NN} \\ &=&
2\sum_{k}[\frac{\partial\eta^{N-1}}{\partial x_{k}}\frac{\partial\eta^{k}}{\partial x_{N-1}} +
\frac{\partial\eta^{N}}{\partial x_{k}}\frac{\partial\eta^{k}}{\partial x_{N}} ],
\end{eqnarray*}
(v)$(\stackrel{\rightarrow}{n},\stackrel{\rightarrow}{n}\cdot\nabla\eta)^2 =
 (\sum_{i,j}(n_{i},n_{j})\frac{\partial\eta^{i}}{\partial x_{j}})^2 = (\frac{\partial\eta^{N-1}}{\partial x_{N-1}}
+\frac{\partial\eta^{N}}{\partial x_{N}})^2$,\\
(vi) $\abs{\stackrel{\rightarrow}{n}\cdot\nabla\eta}^2 = \sum_{i=1}^{N}
\abs{(\sum_{j}\frac{\partial\eta^{j}}{\partial x_{i}}n_{j})}^2 = 
\sum_{i}\left (\abs{\frac{\partial\eta^{N-1}}{\partial x_{i}}}^2+ \abs{\frac{\partial\eta^{N}}{\partial x_{i}}}^2\right)$,\\
(vii) \begin{eqnarray*}(D_{\tau_{i}}\eta)^{\perp} = D_{\tau_{i}} \eta -\sum_{j=1}^{N-2}(\tau_{j}\cdot D_{\tau_{i}} \eta)\tau_{j} &=&
(\frac{\partial\eta^{1}}{\partial x_{i}}, \cdots,\frac{\partial\eta^{N}}{\partial x_{i}}) -\sum_{j\leq N-2}\frac{\partial\eta^{j}}{\partial x_{i}}\tau_{j}\\& =& 
(0,\cdots,0, \frac{\partial\eta^{N-1}}{\partial x_{i}}, \frac{\partial\eta^{N}}{\partial x_{i}}),
\end{eqnarray*}
(viii) $\sum_{i\leq N-2}\abs{(D_{\tau_{i}}\eta)^{\perp}}^2 = \sum_{i\leq N-2}\left (
\abs{\frac{\partial\eta^{N-1}}{\partial x_{i}}}^2+ 
\abs{\frac{\partial\eta^{N}}{\partial x_{i}}}^2\right),$\\
(ix) $\tau_{i}\cdot D_{\tau_{j}}\eta = \frac{\partial\eta^{i}}{\partial x_{j}}$,\\
(x) $\sum_{i,j\leq N-2} (\tau_{i}\cdot D_{\tau_{j}}\eta)(\tau_{j}\cdot D_{\tau_{i}}\eta) =
 \sum_{i,j\leq N-2}\frac{\partial\eta^{i}}{\partial x_{j}}\frac{\partial\eta^{j}}{\partial x_{i}}$.\\
Let 
\begin{align*}
 M&= - \text{trace} ((\nabla \eta)^2) + \abs{\stackrel{\rightarrow}{n}\cdot\nabla\eta}^2 + 
2(\stackrel{\rightarrow}{n},\stackrel{\rightarrow}{n}\cdot(\nabla\eta)^2) - (\stackrel{\rightarrow}{n},\stackrel{\rightarrow}{n}\cdot\nabla\eta)^2,\\
 N&= \sum_{i=1}^{N-2}\abs{(D_{\tau_{i}}\eta)^{\perp}}^2 - \sum_{i,j=1}^{N-2}(\tau_{i}\cdot D_{\tau_{j}}\eta)(\tau_{j}\cdot D_{\tau_{i}}\eta)\\
&=  \sum_{i\leq N-2}\left (
\abs{\frac{\partial\eta^{N-1}}{\partial x_{i}}}^2+ 
\abs{\frac{\partial\eta^{N}}{\partial x_{i}}}^2\right) - \sum_{i,j\leq N-2}\frac{\partial\eta^{i}}{\partial x_{j}}\frac{\partial\eta^{j}}{\partial x_{i}}.
\end{align*}
Observe from (iii) and (iv) that
\begin{multline*}
 - \text{trace} ((\nabla \eta)^2) =
-\sum_{i,k\leq N-2}\frac{\partial\eta^{i}}{\partial x_{k}}\frac{\partial\eta^{k}}{\partial x_{i}}
-2 (\stackrel{\rightarrow}{n},\stackrel{\rightarrow}{n}\cdot(\nabla\eta)^2)\\
+\sum_{N-1\leq i\leq N}\left(\frac{\partial\eta^{i}}{\partial x_{N-1}}\frac{\partial\eta^{N-1}}{\partial x_{i}}
+\frac{\partial\eta^{i}}{\partial x_{N}}\frac{\partial\eta^{N}}{\partial x_{i}}\right).
\end{multline*}
(xi) From (v), (vi) and (x), we have \begin{multline*}
       M-N =
\sum_{N-1\leq i\leq N}\left(\frac{\partial\eta^{i}}{\partial x_{N-1}}\frac{\partial\eta^{N-1}}{\partial x_{i}}
+\frac{\partial\eta^{i}}{\partial x_{N}}\frac{\partial\eta^{N}}{\partial x_{i}}\right) + 
\sum_{i}\left (\abs{\frac{\partial\eta^{N-1}}{\partial x_{i}}}^2+ \abs{\frac{\partial\eta^{N}}{\partial x_{i}}}^2\right)\\
 - (\frac{\partial\eta^{N-1}}{\partial x_{N-1}}
+\frac{\partial\eta^{N}}{\partial x_{N}})^2 -\sum_{i\leq N-2}\left (
\abs{\frac{\partial\eta^{N-1}}{\partial x_{i}}}^2+ 
\abs{\frac{\partial\eta^{N}}{\partial x_{i}}}^2\right)\\
= (\frac{\partial\eta^{N-1}}{\partial x_{N}}
+\frac{\partial\eta^{N}}{\partial x_{N-1}})^2 + (\frac{\partial\eta^{N-1}}{\partial x_{N-1}}
-\frac{\partial\eta^{N}}{\partial x_{N}})^2 \\
= \abs{D_{x_{N-1}, x_{N}} (\eta^{N-1},\eta^{N})}^2 - 2 Jac_{x_{N-1}, x_{N}} (\eta^{N-1},\eta^{N}). 
      \end{multline*}
Thus from (\ref{svreduced}) and (\ref{sve}), we find that
\begin{equation*}
\lim_{\varepsilon\rightarrow 0}\delta^{2}E_{\varepsilon}(u_{\varepsilon},\eta,\zeta) 
= m\pi\delta^{2}E(\Gamma,\eta,\zeta) + 
m\pi\int_{\Gamma} \left(\abs{D_{\perp} (\eta^{\perp})}^2 - 2Jac_{\perp} (\eta^{\perp})\right) d\mathcal{H}^{N-2}.
\end{equation*}
The proof of Theorem \ref{mainSV} is now complete.
\end{proof}
\subsection*
{Acknowledgement.}
This paper grew out of a question that Peter Sternberg and Kevin Zumbrun asked the author in December 2012 about the relation between second inner 
variation and Poincar\'e inequality for area-minimizing surfaces with volume constraint. It is a great pleasure to thank them for 
this intriguing question and other interesting discussions related to the subject of this paper. The author 
is grateful to the anonymous referee for his/her careful reading and constructive comments which resulted in a hopefully
improved version of the original manuscript.

{} 

\end{document}